\documentclass[11pt,letterpaper]{amsart}
\usepackage{graphicx} 
\usepackage{amsfonts}
\usepackage{amsmath}
\usepackage{bbm}
\usepackage{xcolor}
\usepackage{amsthm}
\usepackage{mathrsfs}  
\usepackage{hyperref}
\usepackage{amssymb}
\usepackage{lipsum}
\usepackage{apacite}
\hypersetup{
    colorlinks,
    linkcolor={blue},
    citecolor={blue},
    urlcolor={blue}
}
\usepackage{fmtcount,refcount}
\usepackage{enumerate}   
\usepackage{booktabs}

\newtheorem{theorem}{Theorem}
\newtheorem{lemma}{Lemma}
\newtheorem{prop}{Proposition}
\newtheorem{corr}{Corollary}
\newtheorem{remark}{Remark}
\newtheorem{example}{Example}
\newtheorem{definition}{Definition}
\newtheorem{assum}{Assumption}[section]

\newcommand{\floor}[1]{\left\lfloor#1\right\rfloor}
\newcommand{\ceil}[1]{\left\lceil#1\right\rceil}

\DeclareRobustCommand{\compactright}{
\lhook\joinrel\relbar\joinrel\joinrel\joinrel\lhook\joinrel\rightarrow
}

\title{\textbf{Fractional and Integer Order Sobolev Spaces for Compact Metric Graphs}}
\author[AWADELKARIM]{ELSIDDIG AWADELKARIM}
\author[BOLIN]{DAVID BOLIN}
\author[SIMAS]{Alexandre B. Simas}

\address[AWADELKARIM]{\sc Elsiddig Awadelkarim\\ Statistics Program, Division of Computer, Electrical and Mathematical Sciences and Engineering, King Abdullah University of Science and Technology, Thuwal, KSA.}
\email{elsiddigawadelkarim.elsiddig@kaust.edu.sa}

\address[BOLIN]{\sc David Bolin\\ Statistics Program, Division of Computer, Electrical and Mathematical Sciences and Engineering, King Abdullah University of Science and Technology, Thuwal, KSA.}
\email{david.bolin@kaust.edu.sa}

\address[SIMAS]{\sc Alexandre B. Simas\\ Statistics Program, Division of Computer, Electrical and Mathematical Sciences and Engineering, King Abdullah University of Science and Technology, Thuwal, KSA.}
\email{alexandre.simas@kaust.edu.sa}

\date{\today}
\subjclass[2020]{Primary 35R02; Secondary   35A01 ,   35A02}
\keywords{Sobolev Spaces, Quantum Graphs, Eigenfunctions, Fractional Laplacian}

\begin{document}
\begin{abstract}
Given a compact metric graph $\Gamma$ and the Laplacian $\Delta_{\Gamma}$ coupled with standard (Kirchhoff) vertex conditions, solutions to fractional elliptic partial differential equations of the form $(\kappa^2 - \Delta_{\Gamma})^{\alpha/2}u=f$ on $\Gamma$ exhibit a distinctive regularity structure: even-order derivatives are continuous across vertices, while odd-order derivatives may be discontinuous. This non-standard smoothness property precludes the direct application of classical tools from real functional analysis. Because of this, we introduce and systematically study new families of Sobolev spaces tailored to this setting. We define these spaces, denoted $W^{\alpha,p}(\Gamma)$ and $H^{\alpha}(\Gamma)$, to respect the continuity constraints on even-order derivatives at vertices, while permitting discontinuities in odd-order derivatives. We establish their fundamental properties, including characterizations via interpolation and Sobolev-Slobodeckij norms, embedding theorems into H\"older and Lebesgue spaces, and compactness results. A central contribution in this investigation is the derivation of uniform bounds on the supremum norm of eigenfunctions for a class of Laplacians on metric graphs, a result of independent interest. Finally, we demonstrate that these spaces provide a natural framework for analyzing the regularity of solutions to fractional elliptic PDEs and SPDEs driven by Gaussian white noise on metric graphs, in particular, establishing a general characterization of the domain of the fractional powers of $(\kappa^2-\Delta_{\Gamma})$ and $(\kappa^2-\nabla(a\nabla))$ in terms of the Sobolev spaces we introduce, thereby 
extending all previously known characterizations in the literature, and improving the regularity results previously obtained to their sharp counterparts (with general non-integer fractional powers).
We also show that these spaces are fundamental to the characterization of Gaussian free fields on metric graphs.
\end{abstract}
\maketitle

\section{Introduction}
\label{sec:intro}
Increasingly rich and abundant network data have motivated the development of sophisticated modeling tools capable of capturing spatial and topological structure on graph-like domains. In many such applications, networks are naturally modeled as metric graphs; examples include traffic flow in road networks \cite{damilya}, scientific citation networks \cite{borovitsky}, and nanotechnology \cite{kuchment}. Unlike combinatorial graphs, where edges are abstract connections, metric graphs endow each edge with an intrinsic length, thereby enabling the representation of geometric and physical phenomena. A powerful tool for spatial modeling on these structures is provided by Whittle-Mat\'ern Gaussian random fields, whose efficient generation on metric graphs is enabled by the stochastic partial differential equation (SPDE) approach \cite{bern_gaussian_matern}. Given a compact metric graph $\Gamma$, these Gaussian fields arise as solutions of the SPDE 
\begin{equation}\label{eq:whittle}
    (\kappa^2-\Delta_{\Gamma})^{\alpha/2}u=\mathcal{W},
\end{equation} where $\mathcal{W}$ is a Gaussian white noise on $\mathbb{L}_2(\Gamma)$ (the space of square integrable functions on $\Gamma$), $\Delta_{\Gamma}$ is the edge-wise Laplacian restricted to an appropriate subspace, and the parameter $\alpha>0$ controls the smoothness of the Gaussian field $u$. Solutions to such SPDEs exhibit a characteristic parity-dependent regularity: even-order derivatives remain continuous across graph vertices, while odd-order derivatives may display jump discontinuities. This behavior fundamentally differs from classical Euclidean settings and obstructs the direct use of conventional tools from real functional analysis.
The incompatibility with classical Sobolev theory manifests already at second order. For a compact metric graph $\Gamma$ with edges $\mathcal{E}$, let $C(\Gamma)$ denote the space of real-valued continuous functions on $\Gamma$. The space
$H^1(\Gamma) := \bigoplus_{e\in\mathcal{E}}H^{1}(e)\cap C(\Gamma)$
equipped with the norm $\|u\|_{H^1(\Gamma)}^2 = \sum_{e \in \mathcal{E}} \|u|_e\|_{H^1(e)}^2$,
provides a natural analogue of the one-dimensional Sobolev space $H^1$, by enforcing continuity across vertices. One way to motivate this choice is by considering a squeezing limit of graph-like manifolds \cite{berk_2}. 
Further, by interpolation one has natural candidates for the spaces $H^{\alpha}(\Gamma)$ for $\alpha\in(0,1)$ \cite{bern_gaussian_matern, mathcomp_paper}.
In contrast, no canonical definition for $H^2(\Gamma)$ arises, even when derivatives exist edge-wise, they may suffer jumps, as the derivatives naturally are discontinuous at the vertices of degree larger than $2$ of $\Gamma$. This regularity behavior obstructs direct extension of Euclidean $H^2$ theory.
A complementary perspective comes from Gaussian free fields (GFFs). The massless GFF on a metric graph can be constructed by extending the discrete GFF on the combinatorial graph by assigning independent Brownian bridges to each edge \cite{lupu_2016}. The covariance function of the GFF is the Green's function of the Laplacian coupled with Dirichlet boundary conditions, both in the discrete and the continuum limiting case, equivalently their covariance operator is the inverse of the Dirichlet Laplacian $(-\Delta)^{-1}$. Their massive counterpart, called massive GFF, are defined as the centered Gaussian field with covariance operator $(\kappa^2 - \Delta)^{-1}$ with $\kappa\not=0$ \cite{sheffield_2007, rodriguez_2016}. The definition of massive GFF coincides with the definition of Whittle-Mat\'ern Gaussian fields on metric graphs given in \citeA{bern_gaussian_matern} in the case $\alpha = 1$. Furthermore, the Cameron-Martin space associated with the massive GFF is the energetic space of the operator $L=\kappa^2 - \Delta$. As we shall demonstrate, this space can be characterized using the Sobolev spaces $H^\alpha(\Gamma)$ we construct in this article.

In this article, we introduce and systematically study new families of Sobolev spaces tailored to the regularity structure induced by such equations on metric graphs. These spaces, denoted by $W^{\alpha,p}(\Gamma)$  and $H^{\alpha}(\Gamma)$, are defined by enforcing continuity of even-order derivatives at vertices while permitting discontinuities in odd-order derivatives. We extend these definitions to fractional orders using interpolation theory and the Sobolev-Slobodeckij norm, and establish rigorous characterizations, embedding theorems, and equivalences between the two classes. A central contribution of this work is proving Sobolev embedding results, including compact embeddings into H\"older and Lebesgue spaces. Furthermore, in our endeavor to demonstrate the applicability of these spaces, we derive uniform bounds on the supremum norm of eigenfunctions of the appropriate Laplacians on metric graphs, which is a result of independent interest and importance. 
We further show that the domains of fractional powers of the operator $(\kappa^2-\Delta_{\Gamma})^{\alpha/2}$ and $(\kappa^2-\nabla(a\nabla))^{\alpha/2}$ can be expressed in terms of the introduced Sobolev spaces,
establishing them as the natural function spaces for the analysis of fractional elliptic PDEs on metric graphs. Characterizations of Sobolev-type spaces as domains of fractional operators on metric graphs have been explored in \citeA{bern_gaussian_matern, mathcomp_paper, markov_paper}. In this work, we generalize those results by employing the Sobolev spaces $W^{\alpha,p}(\Gamma)$  and $H^{\alpha}(\Gamma)$ introduced here. Although the definitions of the spaces $W^{\alpha,p}(\Gamma)$  and $H^{\alpha}(\Gamma)$ do not require $\Gamma$ to be compact and extend to the non-compact setting, the characterizations established in this work rely on the compactness assumption of $\Gamma$, and we therefore assume compactness throughout to simplify the notation. Finally, we study the applications of the introduced Sobolev spaces to operators of the form $(\kappa^2-\nabla(a\nabla))^{\alpha/2}$ where $\kappa^2$ and $a$ are functions. These are important as they arise naturally through the study of generalized Whittle-Mat\'ern fields \cite{mathcomp_paper, lindgren_bolin_rue_2022}.\\
The paper is organized as follows. In Section~\ref{sec:setting}, we introduce and discuss the mathematical setting and notation needed for metric graphs. In Section~\ref{sec:sobolev_definition}, we define the $W$  and $H$  spaces: the former via the Sobolev-Slobodeckij norm, and the latter via interpolation of integer order Sobolev spaces. We prove characterization theorems, Sobolev embeddings, and the key equivalence $H^{\alpha}(\Gamma)\cong W^{\alpha,2}(\Gamma)$. Section~\ref{sec:eigen_bounds} is devoted entirely to the uniform eigenfunction bounds, an important result used throughout the paper. In Sections~\ref{sec:appli_1} and~\ref{sec:appli_2} we apply these spaces to characterize the regularity of solutions to the fractional elliptic PDEs $(\kappa^2-\Delta_{\Gamma})^{\alpha/2}u=f$ and $(\kappa^2-\nabla(a\nabla u))^{\alpha/2}u=f$ where in the latter equation $\kappa$ and $a$ are functions. We also elucidate the natural connection between these Sobolev spaces and Gaussian free fields on metric graphs.

\section{The Setting and Notation}\label{sec:setting}
\subsection{Compact Metric Graphs} A compact metric graph $\Gamma$ is a compact metric space constructed from a finite graph $\bar{\mathcal{G}}=(\bar{\mathcal{V}},\bar{\mathcal{E}})$, where $\bar{\mathcal{V}}$ denotes the set of vertices and $\bar{\mathcal{E}}$ the set of edges, together with a collection of positive real numbers $\{l_{\bar{e}}\}_{\bar{e}\in\bar{\mathcal{E}}}$ representing the lengths of the edges $\bar{e}\in\bar{\mathcal{E}}$. Formally, each edge $\bar{e}\in\bar{\mathcal{E}}$ is identified with the closed interval $[0,l_{\bar{e}}]$, and the graph $\bar{\mathcal{G}}$ prescribes how these intervals are glued together at their endpoints to form the space $\Gamma$.

The mathematical construction of $\Gamma$ is as follows. We begin by defining the disjoint union of intervals: $$\bar{\Gamma} = \coprod_{\bar{e} \in \bar{\mathcal{E}}} [0, l_{\bar{e}}]$$ 
where, for any subset $A \subseteq \bar{\mathcal{E}}$  and family of sets $\{S_{\bar{e}}\}_{\bar{e} \in A}$, the disjoint union is defined as
$$
\coprod_{\bar{e} \in A} S_{\bar{e}} = \bigcup_{\bar{e} \in A} \{(\bar{e}, s) : s \in S_{\bar{e}}\}.
$$

Let $\pi_{\bar{\mathcal{V}}} : \coprod_{\bar{e} \in \bar{\mathcal{E}}} \{0, l_{\bar{e}}\} \to \bar{\mathcal{V}}$ be a surjective map that satisfies $\pi_{\bar{\mathcal{V}}}^{-1}(\bar{v})\subset \coprod_{\bar{e}\in\bar{\mathcal{E}}_{\bar{v}}}\{0,l_{\bar{e}}\}$ for every $\bar{v}\in\bar{\mathcal{V}}$ where $\bar{\mathcal{E}}_{\bar{v}} = \{\bar{e}\in\bar{\mathcal{E}} : \bar{v}\in \bar{e}\}$ is the set of edges in $\bar{\Gamma}$ incident to $\bar{v}$. The map $\pi_{\bar{\mathcal{V}}}$ assigns each endpoint of an interval to its corresponding vertex in $\bar{\mathcal{G}}$. We define the equivalence relation $\sim$ on $\bar{\Gamma}$ by $a\sim b$ for $a,b\in\bar{\Gamma}$ if either $i)$ $a=b$ or $ii)$ $\pi_{\bar{\mathcal{V}}}(a)=\pi_{\bar{\mathcal{V}}}(b)$ and $a,b\in \coprod_{\bar{e} \in \bar{\mathcal{E}}} \{0, l_{\bar{e}}\}$.
The metric graph $\Gamma$ is then defined as the quotient space 
$$\Gamma = \bar{\Gamma}/\sim.$$ 
The space $\Gamma$ is equipped with the quotient topology induced by the natural projection $\pi:\bar{\Gamma}\to\Gamma$ resulting from the quotient, i.e., the function that maps points $a$ in $\bar{\Gamma}$ to their equivalence class $\pi(a)$. Denote the set of vertices of $\Gamma$ by $\mathcal{V}= \pi(\pi^{-1}_{\bar{\mathcal{V}}}(\bar{\mathcal{V}}))$ and the set of edges $\mathcal{E}=\{\pi(\{\bar{e}\}\times[0,l_{\bar{e}}]): \bar{e}\in\bar{\mathcal{E}}\}$. With this definition every edge $e\in\mathcal{E}$ has a corresponding edge $\bar{e}\in\bar{\mathcal{E}}$. For every edge $e\in\mathcal{E}$, let $\bar{e}\in\bar{\mathcal{E}}$ it corresponding edge in $\bar{\mathcal{E}}$ denote $l_e=l_{\bar{e}}$ and define $\partial e=\pi(\{\bar{e}\}\times\{0,l_{\bar{e}}\})\subset \mathcal{V}$ to be the set of vertices at the end of $e$ (boundary of $e$ or vertices incident to $e$). Observe that for $s\in\Gamma\backslash\mathcal{V}$, $\pi(s)$ is a simpleton, and if $s\in\mathcal{V}$, then $\#\pi_{\bar{\mathcal{V}}}(s) = \mathrm{deg}(\pi_{\bar{\mathcal{V}}}(s))$, where $\mathrm{deg}(\pi_{\bar{\mathcal{V}}}(s))$ is the degree of the vertex $\pi_{\bar{\mathcal{V}}}(s)$ in the combinatorial graph $\bar{\mathcal{G}}$. Thus, we define $\mathrm{deg}(s) = \#\pi_{\bar{\mathcal{V}}}(s)$ for $s\in\mathcal{V}$. We call the maps $\zeta_1,\zeta_2:[0,l_e]\to e$ defined by 
$$\zeta_1(t)=\pi((\bar{e},t)), \text{  and   }\zeta_2(t)=\zeta_1(l_e-t),$$ the \emph{natural parameterizations} of $e$. Observe that the natural parameterizations correspond to the two possible directions of the edge $e$. These maps induce a metric on $e$ by $d(x,y)=\inf\{|\tilde{x}-\tilde{y}|:\tilde{x}\in\zeta_1^{-1}(x),\tilde{y}\in\zeta_1^{-1}(y)\}$, and $e$ equipped with this metric is isometrically isomorphic to either the interval $[0,l_e]$ or the circle $\mathbb{S}^1_{l_e}=\{(x,y)\in\mathbb{R}^2:x^2+y^2=l_e^2\}$. Define the set of possible directions at $v\in\mathcal{V}$ by
\begin{equation*}
    \textrm{Dir}(v) = \left\{ (e,\zeta) : e\in\mathcal{E},\zeta \textrm{ is a parameterization of } e \textrm{ such that } \zeta(0)=v\right\},
\end{equation*}
clearly $\#\textrm{Dir}(v) = \textrm{deg}(v)$. A path on $\Gamma$ is a continuous map $p:[0,l]\to \Gamma$ and satisfies that $p|_{p^{-1}(e\backslash\partial e)}:p^{-1}(e\backslash\partial e)\to e\backslash \partial e$ is a local isometry for every $e\in\mathcal{E}$. Clearly, for a path $p$ the set $p^{-1}(\mathcal{V})$ is finite, and we call $l$ the length of $p$. The topology of $\Gamma$ is metrizable with the metric $$d(x,y) = \inf\left\{l: l>0, p:[0,l]\to\Gamma, p(0)=x, p(l)=y\right\},$$
which is known as the geodesic distance on $\Gamma$.

\subsection{Additional Notation and Derivatives on $\Gamma$} In Section~\ref{sec:intro} we noted that solutions to Equation \eqref{eq:whittle} exhibit discontinuity of odd-order derivatives. Consequently, these derivatives cannot be meaningfully defined as functions on $\Gamma$. However, they arise naturally as well-defined functions on $\bar{\Gamma}$. In this section, we establish the notion of derivative on $\Gamma$ and $\bar{\Gamma}$, identifying the two notions of derivatives when they are consistent.

We first introduce the notions of parameterization and reparameterization. A parameterization of the edges of $\Gamma$ (or simply a parameterization of $\Gamma$) is a collection $\eta=\{\eta_e\}_{e\in\mathcal{E}}$, where each $\eta_e$ is a natural parameterization of the edge $e$, that is, we choose a direction for each edge (note that we have $2^{\#\mathcal{E}}$ possible parameterizations).
Such a choice provides a global identification of the edges of $\Gamma$ with the edges of $\bar{\Gamma}=\coprod_{\bar{e} \in \bar{\mathcal{E}}} [0, l_{\bar{e}}]$ ignoring the vertex identifications encoded by $\pi$. Given a parameterization $\eta$ and a function $f:\bar{\Gamma}\to\mathbb{R}$, we define the reparameterization of $f$ under $\eta$ as the function $f_{\eta}:\bar{\Gamma}\to\mathbb{R}$ defined by
$$f_{\eta}((\bar{e},t)) = f\left(\Big(\bar{e},\lim_{\epsilon\to0^+}\eta_e^{-1}\Big\{\pi\big[(\bar{e},t - \textrm{sign}(t-\epsilon)\epsilon)\big]\Big\}\Big)\right).$$
The limit in the expression above is a technical device to allow $t\in[0,l_e]$ in the case $\#\partial e=1$, but for $\#\partial e=2$ it is not needed and one can directly set $\epsilon=0$. Note that the reparameterization either keeps the direction of $e$ or flips the direction. Therefore, it is an involution operation; that is $$(f_{\eta})_{\eta} = f,\quad \forall f:\bar{\Gamma}\to\mathbb{R}.$$
For each edge $\bar{e}\in\bar{\mathcal{E}}$ define $f_{\bar{e}}:[0,l_{\bar{e}}]\to\mathbb{R}$ by $f_{\bar{e}}(t)=f((\bar{e},t))$.
If a function $f:\bar{\Gamma}\to\mathbb{R}$ satisfies $\#f(\pi^{-1}(x))=1$ for every $x\in\Gamma$ we identify $f$ with the well-defined function $f\circ\pi^{-1}:\Gamma\to\mathbb{R}$. Every function $f:\Gamma\to\mathbb{R}$ can be extended to a function $\bar{f}$ defined on $\bar{\Gamma}$ by $\bar{f}=f\circ\pi$. Therefore, all the notions defined previously apply to functions on $\Gamma$ and produce functions defined on $\bar{\Gamma}$.\\

Next, we define the notions of derivatives we will use in this article. Let $f\in\bar{\Gamma}\to\mathbb{R}$ and let $\eta=\{\eta_e\}_{e\in\mathcal{E}}$ be a parameterization of $\Gamma$. For $e\in\mathcal{E}$ and $t\in(0,l_{e})$ define the weak derivative given $\eta$ as $D_{\eta}^kf:\bar{\Gamma}\to\mathbb{R}$ by 
$$D_{\eta}f((\bar{e},t)) = \left((f_{\eta})_{\bar{e}}\right)'(t),$$
where $f'$ denotes the weak derivative of the real function $f$.
Let $\epsilon>0$, and let $f:[0,\epsilon)\to\mathbb{R}$. We define the right lateral derivative at $0$ by $$\partial_+f(0) = \lim_{h\to0^+}(f(h)-f(0))/h,$$ provided that the limit exists. Higher-order right lateral derivatives are defined recursively for $k\in\mathbb{N}$ by $\partial_+^kf=\partial_+^{k-1}(\partial_+f)$, we denote $\partial_+^0f=f$. Analogously, define the left lateral derivatives for functions defined on $(-\epsilon,0]$. For $f:\Gamma\to\mathbb{R}$, $e\in\mathcal{E}$, $v\in\partial e$, and a parametrization $\zeta$ of $e$ that satisfies $\zeta(0)=v$ (i.e. $(e,\zeta)\in\textrm{Dir}(v)$), define the directional derivative of $f$ at $v$ in the direction $(e,\zeta)$ as
$$\partial_{e,\zeta}f(v)=\partial_+ (f\circ\zeta)(0).$$
The outward derivative at $v$ in the direction $(e,\zeta)$ is defined by $\bar{\partial}_{e,\zeta}f(v) = -\partial_{e,\zeta}f(v)$.
When $e$ is not a cycle (i.e., $\#\partial e=2$) there is a unique parameterization $\zeta$ of $e$ satisfying $\zeta(0)=v$, in this case, we write simply $\partial_{e}f(v)=\partial_{e,\zeta}f(v)$ and refer to it as the inward derivative at $v$ along $e$. Similarly for the outward derivative $\bar{\partial}_{e,\zeta}f(v)$, in the case $\#\partial e=2$, we write $\bar{\partial}_{e}f(v)=\bar{\partial}_{e,\zeta}f(v)$ and refer to it as the outward derivative.
When $f$ is regular, denote the derivatives at the endpoints $(\bar{e},0)$ and $(\bar{e},l_e)$, corresponding to the vertices in $\partial e$, by 
$$D_{\eta}f((\bar{e},0)) = \partial_+(f_{\eta})_{\bar{e}}(0),\; \text{and } D_{\eta}f((\bar{e},l_{e})) = \partial_-(f_{\eta})_{\bar{e}}(l_{e}).$$
Higher order derivatives are defined analogously.  An easy and crucial observation is that $D^2$ is invariant under reparameterization; more precisely, $D^2_{\eta}f = D^2_{\tilde{\eta}}f$ for any parameterizations $\eta$ and $\tilde{\eta}$ of $\Gamma$; therefore we only write $D^{2k}$ for the even derivatives without specifying the parameterization. For a function $f:\Gamma\to\mathbb{R}$ we define its derivative given $\eta$ as $D_{\eta}^kf:\bar{\Gamma}\to\mathbb{R}$ by $D_{\eta}^kf=D_{\eta}^k(f\circ\pi)$. The map $\pi$ induces a parameterization $\{\pi_e\}_{e\in\mathcal{E}}$, denote $\nabla f = D_{\pi}$f and $\Delta f = \nabla^2f=D^2f$. Table~\ref{tab:derivatives} summarizes the introduced notation relating to derivatives on metric graphs.\\

\begin{table}[ht]
\centering
\small
\begin{tabular}{p{1.5cm}p{10.3cm}}
\toprule
\textbf{Notation} & \textbf{Definition and Remarks} \\
\midrule
$D_\eta f$ &  $D_\eta f((\bar{e},t)) = \big((f_\eta)_{\bar{e}}\big)'(t)$\\ & Weak derivative w.r.t.\ parameterization $\eta$\\
\addlinespace

$D_\eta^k f$  & Recursively $D_\eta^k f = D_\eta(D_\eta^{k-1} f)$\\ & Higher-order weak derivatives ($\eta$-dependent)\\
\addlinespace

$\nabla f$ &  $\nabla f = D_\pi f$ \\ & Derivative w.r.t. the parameterization induced by the projection $\pi$.\\
\addlinespace

$D^{2k} f$ & $D^{2k}_\eta f$ \\ & Independent of $\eta$\\
\addlinespace

$\Delta f$ &  $\Delta f = D^2 f = \nabla^2 f$\\ & Laplacian on the metric graph, well-defined without reference to $\eta$. \\
\addlinespace

$\partial_{e,\zeta} f(v)$ &  Inward derivative of $f$ at the vertex $v$ in the direction $(e,\zeta)$. \\
\addlinespace

$\partial_e f(v)$ &  $\partial_e f(v) = \partial_{e,\zeta} f(v)$ when the edge $e$ is not a loop.\\
\addlinespace

$\bar{\partial}_{e,\zeta} f(v)$ &  Outward derivative of $f$ at the vertex $v$ in the direction $(e,\zeta)$.\\
\addlinespace

$\bar{\partial}_e f(v)$ &  $\bar{\partial}_e f(v) = \bar{\partial}_{e,\zeta} f(v)$ when the edge $e$ is not a loop \\
\addlinespace
\bottomrule
\addlinespace
\end{tabular}
\caption{Summary of derivative notations for $f:\bar{\Gamma}\to\mathbb{R}$. All derivatives for $f:\Gamma \to \mathbb{R}$ are defined via lift to $\bar{f} = f \circ \pi$ on $\bar{\Gamma}$.}
\label{tab:derivatives}
\end{table}

\subsection{$\mathbb{L}_p$ and H\"older Spaces}
Now, we define natural extensions of standard functional spaces. For $p\in(0,\infty)$ let $\mathbb{L}_p(0,l)$ denote the space of measurable functions $f:(0,l)\to\mathbb{R}$ that satisfy $$\|f\|_{\mathbb{L}_p(0,l)}=\left(\int_0^l|f(x)|^pdx\right)^{1/p}<\infty,$$
and for $p=\infty$ let $\mathbb{L}_{\infty}(0,l)$ be the space of measurable functions that satisfy
$$\|f\|_{\mathbb{L}_{\infty}(0,l)} = \inf \{x\geq0: f^{-1}(\mathbb{R}\backslash[-x,x]) \textrm{ is of measure } 0\}.$$
For a function $f:e\to\mathbb{R}$ we define $$\int_e f(x)dx = \int_{0}^{l_e}(f\circ \zeta)(x)dx$$ for any natural parameterization $\zeta$ of $e$. Note that this definition is well-defined. Indeed, there are only $2$ parameterizations, let $\tilde{\zeta}$ be the second parameterization, by substitution $y=(\zeta^{-1}\circ\tilde{\zeta})(x)$ we have 
\begin{equation*}
    \int_{0}^{l_e}f(\circ\zeta)(x)dx = \int_0^{l_e}(f\circ\tilde{\zeta})(y)dy.
\end{equation*}

Let $\mathbb{L}_p(e)$ be the space of functions on $e$ that satisfy $f\circ\zeta\in\mathbb{L}_p(0,l_e)$, equipped with the norm $\|f\|_{\mathbb{L}_p(e)} = \|f\circ\zeta\|_{\mathbb{L}_p(0,l_e)}$. For $f:\Gamma\to\mathbb{R}$, we set $$\int_{\Gamma}f(x)dx = \sum_{e\in\mathcal{E}}\int_ef(x)dx,\quad \|f\|_{\mathbb{L}_p(\Gamma)} = \left(\int_{\Gamma}|f(x)|^pdx\right)^{1/p},$$
$$\|f\|_{\mathbb{L}_{\infty}(\Gamma)} = \max_{e\in\mathcal{E}}\|f_e\|_{\mathbb{L}_{\infty}(e)}.$$
These definitions yield the direct sum decomposition
$$\mathbb{L}_p(\Gamma)=\bigoplus_{e\in\mathcal{E}}\mathbb{L}_p(e).$$
For $k\in\mathbb{N}_0$ and an interval $I\subset\mathbb{R}$, $C^{k}(I)$ is the space of functions $f:I\to\mathbb{R}$ whose $k$-th derivative is defined and continuous on $I$, where the derivative at the boundary points of $I$ is defined as lateral derivatives. For $\lambda\in(0,1]$ define $C^{k,\lambda}(I)\subset C^{k}(I)$ as the subspace of functions whose $k$-th derivative is $\lambda$-H\"older continuous, and set $C^{k,0}(I)=C^k(I)$. For an edge $e$ with different endpoints (i.e., $\#\partial e=2$) define $C^{k,\lambda}(e)$ as the space of functions $f:e\to\mathbb{R}$ that satisfy $f\circ\zeta\in C^{k,\lambda}([0,l_e])$ for some (and hence every) natural parameterization $\zeta$. In the case where $\#\partial e=1$ define $C^{k,\lambda}(e)$ to be the space of functions $f:e\to\mathbb{R}$ that satisfy $f\circ\zeta\in C^{k,\lambda}([0,l_e])$ and $f\circ\xi\in C^{k,\lambda}([-l_e/2,l_e/2])$, where $\zeta$ is a natural parameterization of $e$ and $\xi:[-l_e/2,l_e/2]\to\mathbb{R}$ is the map defined by
$$\xi(t)=\begin{cases}
    \zeta(t), \quad t\in[0,l_e/2],\\ \zeta(l_e+t), \quad t\in[-l_e/2,0].
\end{cases}$$
\begin{definition}
    Let $C^{(0),0}(\Gamma)=C^{(0)}(\Gamma)=C(\Gamma)$, and for $\gamma\in(0,1]$ define $C^{(0),\gamma}(\Gamma)$ to be the space of $\gamma$-H\"older continuous functions on $\Gamma$. For $k\in\mathbb{N}_0$ and $\gamma\in[0,1]$ define the H\"older space $C^{(k),\gamma}(\Gamma)$ to be space of functions $\varphi\in\bigoplus_{e\in\mathcal{E}}C^{(k),\gamma}(e)$ that satisfy $D^{2j}\varphi\in C^{(0),\gamma}(\Gamma)$ for every $j\in\{0,1,\cdots,\floor{\frac{k}{2}}\}$. To simplify the notation we write $C^{(k)}(\Gamma)=C^{(k),0}(\Gamma)$.
\end{definition}
Note that if $\varphi\in C^{(k),\gamma}(\Gamma)$ then the even-order derivative $D^{2j}\varphi\in C^{0,1}(\Gamma)$ for every \mbox{$j\in\{0,1,\dots,2\floor{\frac{k}{2}}\}\cap[0,k)$}. Indeed, since we have $D^{2j}\varphi\in C(\Gamma)$ and \mbox{$D^{2j+1}_{\eta}\varphi\in\bigoplus_{e\in\mathcal{E}}C(e)$} for any parameterization $\eta$, the mean value theorem along shortest paths implies Lipschitz continuity: for any $x,y\in\Gamma$, let $v_1,v_2,\dots,v_k\in\mathcal{V}$ be a sequence of vertices along a shortest path between $x$ and $y$. Denote $v_0=x$ and $v_{k+1}=y$. We have
\begin{equation}\label{eq:lip_proof}
    \begin{split}
        |D^{2j}\varphi(x)-D^{2j}\varphi(y)|&\leq \sum_{i=0}^k|D^{2j}\varphi(v_i)-D^{2j}\varphi(v_{i+1})|\\&\leq |\mathcal{V}|\sup_{z\in\Gamma}|D^{2j+1}_{\eta}\varphi(z)|d(x,y).
    \end{split}
\end{equation}
Which shows that $D^{2j}\varphi\in C^{0,1}(\Gamma)$.

\section{Sobolev Spaces}
\label{sec:sobolev_definition}
In this section, we introduce the Sobolev spaces $W^{\alpha,p}(\Gamma)$ and $H^{\alpha}(\Gamma)$ and establish their key characterizations. While the definitions of these spaces do not require $\Gamma$ to be compact and extend to the non-compact graphs, the characterizations presented here rely on the compactness assumption of $\Gamma$, and we therefore assume compactness throughout to simplify the notation. For normed spaces $(V_1,\|\cdot\|_1)$ and $(V_2,\|\cdot\|_2)$, we write $V_1\hookrightarrow V_2$ to denote a continuous embedding, i.e., $V_1\subset V_2$ and the inclusion map $\iota:V_1\to V_2$ is continuous. If, in addition, the inclusion $\iota$ is compact, we write $V_1 \compactright V_2$. When continuous embedding holds in both directions ($V_1\hookrightarrow V_2$ and $V_2\hookrightarrow V_1$), we write $V_1\cong V_2$, and we say that $V_1$ and $V_2$ are isomorphic.
\subsection{$W$ Spaces}
For $(\alpha,p)\in\mathbb{R}_+^2$ and an edge $e$, we define the space $W^{\alpha,p}(e)$ by pushing forward the functions in the classical fractional Sobolev space $W^{\alpha,p}(0,l_e)$ on the interval $[0,l_e]$ via the natural parameterizations of $e$. We now give a precise definition.
\begin{definition}
For $(\alpha,p)\in\mathbb{N}\times [1,\infty)$ and $e\in\mathcal{E}$, define $W^{\alpha,p}(e)$ as the space of functions $\varphi\in \mathbb{L}_p(e)$ such that for every $j\in\{0,\dots,\alpha\}$ the weak derivative $D^j\varphi$ exists and is in $\mathbb{L}_p(e)$. For $\alpha\in(0,1)$ define $W^{\alpha,p}(e)$ to be the space of functions $\varphi\in \mathbb{L}_p(e)$ that satisfy $|\varphi|_{\alpha,p,e}<\infty$, where
$$|f|_{\gamma,p,e}^p = \int_{e}\int_{e} \frac{|f(x)-f(y)|^p}{d(x,y)^{1+\gamma p}}dxdy.$$
For $\alpha\in(1,\infty)\backslash\mathbb{N}$ define $W^{\alpha,p}(e)$ as the space of functions $\varphi\in W^{\floor{\alpha},p}(e)$ such that $D^{\floor{\alpha}}\varphi\in W^{\alpha-\floor{\alpha},p}(e)$. We equip $W^{\alpha,p}(e)$ with the norm $$\|f\|_{W^{\alpha,p}(e)}^p = \sum_{j=0}^{\floor{\alpha}}\|D^{j}f\|_{\mathbb{L}_p(e)}^p + |D^{\floor{\alpha}}f|_{\alpha-\floor{\alpha},p,e}^p.$$
We call the norm $\|.\|_{W^{\alpha,p}(e)}$ the Sobolev-Slobodeckij norm.
\end{definition}
Regularity results such as \citeA[Theorem~3 and Example~1]{bern_gaussian_matern} demonstrate that patching together the edge-wise spaces $W^{\alpha,p}(e)$ to construct general spaces that capture the regularity structures requires care. In Section~\ref{sec:appli_1} we define the fractional power operator $(\kappa^2-\Delta_{\Gamma})^{\beta}$, but for the sake of motivating the next definitions, we allow ourselves to use it in this paragraph without defining it. Our goal is to define a space $W^{\alpha,p}(\Gamma)\subset \bigoplus_{e\in\mathcal{E}}W^{\alpha,p}(e)$ that can be used to describe the regularity of the solutions of elliptic partial differential equations on metric graphs. In particular, if $f\in W^{\alpha,2}(\Gamma)$ and $u$ is a solution of the PDE $(\kappa^2-\Delta_{\Gamma})^{\beta}u=f$, then we expect $u\in W^{\alpha+2\beta,2}(\Gamma)$. It is well-known that a similar property holds for the spaces $W^{\alpha,2}(e)$ (see \citeA[Theorem 8.13]{gilbarg_trudinger} and \citeA[Proposition 5.38]{demengel}). The following example illustrates why one cannot directly define $W^{\alpha,p}(\Gamma)$ by requiring smoothness along paths.
\begin{example}
    Let $u\in\bigoplus_{e\in\mathcal{E}}C^{\infty}(e)$, and suppose we require that $u\circ p\in C^{1}([0,l])$ for every path $p:[0,l]\to\Gamma$. Then, at any vertex $v$ of degree greater than $2$, all derivatives of $u$ must vanish. Indeed, let $v$ be a vertex incident to three distinct edges $e_1,e_2,e_3$, and let $\zeta_1,\zeta_2,\zeta_3$ be natural parameterizations of the these edges such that $\zeta_i(0)=v$. Define the paths:
$$p_1(t)=\begin{cases}
    \zeta_2(-t),\; t\in[-l_{e_2},0]\\
    \zeta_1(t),\; t\in[0,l_{e_1}]
\end{cases}\quad p_2(t)=\begin{cases}
    \zeta_2(-t),\; t\in[-l_{e_2},0]\\
    \zeta_3(t),\; t\in[0,l_{e_3}]
\end{cases}$$
$$p_3(t)=\begin{cases}
    \zeta_3(-t),\; t\in[-l_{e_3},0]\\
    \zeta_1(t),\; t\in[0,l_{e_1}].
\end{cases}$$
We assume that $(u\circ p_1)'(0)=(u\circ p_2)'(0)=(u\circ p_3)'(0)$. Now, observe that $\partial_+(u\circ p_1)(0)=\partial_{e_1}u$ and $\partial_-(u\circ p_1)(0)=-\partial_{e_2}u$ hence $\partial_{e_1}u=-\partial_{e_2}u$. Similarly, from $p_2$, we get $\partial_{e_3}u=-\partial_{e_2}u$, and from $p_3$, we get $\partial_{e_1}u=-\partial_{e_e}u$. Combining these three equalities yields $$\partial_{e_1}u=\partial_{e_2}u=\partial_{e_3}u=0.$$
This forces all first derivatives to vanish at any vertex of degree at least three, a condition too restrictive to model general solutions of PDEs.
\end{example}
The previous example and \citeA[Theorem~3]{bern_gaussian_matern} motivate our definition of Sobolev spaces on metric graphs. The following definition generalizes the Sobolev-Slobodeckij characterization of fractional Sobolev spaces on intervals (cf.~\citeA{di_nezza}) to metric graphs.

\begin{definition}\label{def:W_Gamma_1}
For $(\alpha,p)\in\mathbb{N}\times [1,\infty)$, define $W^{\alpha,p}(\Gamma)$ as the space of functions $\varphi\in \bigoplus_{e\in\mathcal{E}}W^{\alpha,p}(e)$ such that for every $j\in\{0,\dots,\floor{\frac{\alpha-1}{2}}\}$ the weak derivative $D^{2j}\varphi$ exists and is continuous on $\Gamma$. For $\alpha\in(0,1)$, define $W^{\alpha,p}(\Gamma)$ as the space of functions $\varphi\in \mathbb{L}_p(\Gamma)$ that satisfy $|\varphi|_{\alpha,p}<\infty$, where
$$|f|_{\gamma,p}^p = \int_{\Gamma}\int_{\Gamma} \frac{|f(x)-f(y)|^p}{d(x,y)^{1+\gamma p}}dxdy.$$
\sloppy For $\alpha\in(1,\infty)\backslash\mathbb{N}$ define $W^{\alpha,p}(\Gamma)$ to be the space of functions $\varphi\in \bigoplus_{e\in\mathcal{E}}W^{\alpha,p}(e)\cap W^{\floor{\alpha},p}(\Gamma)$ that satisfy $D^{\floor{\alpha}}\varphi\in W^{\alpha-\floor{\alpha},p}(\Gamma)$ if $\floor{\alpha}$ is even.
We equip $W^{\alpha,p}(\Gamma)$ with the norm $$\|f\|_{W^{\alpha,p}(\Gamma)}^p = \sum_{e\in\mathcal{E}}\|f\|_{W^{\alpha,p}(e)}^p + |D^{2\floor{\floor{\alpha}/2}}f|_{\alpha-\floor{\alpha},p}^p.$$
For $p=\infty$ and $\alpha\in\mathbb{N}$ define $W^{\alpha,\infty}(\Gamma)$ as the space of functions $\varphi\in \bigoplus_{e\in\mathcal{E}} W^{\alpha,\infty}(e)$ such that for every $j\in\{0,\dots\floor{\frac{\alpha-1}{2}}\}$ the weak derivative $D^{2j}\varphi$ exists and is continuous on $\Gamma$.
\end{definition}

Although the path formulation is not appropriate to define the Sobolev spaces, it is valid for $\alpha\in(0,1]$, where an equivalent definition of $W^{\alpha,p}(\Gamma)$ is given by requiring that the pullback along every path lies in the corresponding fractional Sobolev spaces. That is, let $\mathcal{P}(\Gamma)$ denote the set of all paths on $\Gamma$. For each path $\zeta$ define the pullback operator $\Xi_{\zeta}(\varphi)=\varphi\circ\zeta$. Then:
$$W^{\alpha,p}(\Gamma) = \bigcap_{\zeta\in\mathcal{P}(\Gamma)}\Xi_{\zeta}^{-1}\left(W^{\alpha,p}(\textrm{Domain}(\zeta))\right).$$
The theory developed later in this section establishes that this path-based characterization holds for $\alpha\in\left(0,1+\frac{1}{p}\right)$.

Next, we establish characterizations of the space $W^{\alpha,p}(\Gamma)$  by deriving key estimates on functions $\varphi$ in this space. As is standard in such arguments, constants appearing in inequalities may change from line to line; we do not track these changes explicitly and denote all such constants uniformly by $C$, as the essential point is their independence from the function $\varphi$. The constant $C$ depends only on $\alpha$, $p$, and the metric graph $\Gamma$. 

\begin{theorem}\label{thm:W_characterization_1}
    Let $p>1$ and $\alpha\in\left(0,\frac{1}{p}\right)$. We have $$W^{\alpha,p}(\Gamma)\cong\bigoplus_{e\in\mathcal{E}}W^{\alpha,p}(e).$$
\end{theorem}
\begin{proof}
The inclusion $W^{\alpha,p}(\Gamma)\hookrightarrow \bigoplus_{e\in\mathcal{E}}W^{\alpha,p}(e)$ is immediate from the definition of $W^{\alpha,p}(\Gamma)$. For the reverse inclusion, let $\varphi\in\bigoplus_{e\in\mathcal{E}}W^{\alpha,p}(e)$. We only need to show that $|\varphi|^p_{\alpha,p}<C\left(\sum_{e\in\mathcal{E}}|\varphi|^p_{\alpha,p,e} + \|\varphi\|_{\mathbb{L}_p(e)}^p\right)$ for a constant $C$ independent of $\varphi$, as the inequality for the other part of the theorem is trivial. By definition, we have
 \begin{equation}\label{eq:int_int_sum_1}
     \int_\Gamma\int_\Gamma \frac{|\varphi(x)-\varphi(y)|^p}{d(x,y)^{1+\alpha p}}dxdy = \sum_{(e_1,e_2)\in\mathcal{E}^2}\int_{e_1}\int_{e_2} \frac{|\varphi(x)-\varphi(y)|^p}{d(x,y)^{1+\alpha p}}dxdy.
 \end{equation}
If $e_1=e_2$ in the sum then the integral is precisely $|\varphi|_{\alpha,p,e_{1}}^p$. Assume $e_1\not=e_2$. If $e_1\cap e_2=\varnothing$ the quantity $d(x,y)$ in \eqref{eq:int_int_sum_1} is bounded away from zero, and using the inequality $|x+y|^p\leq 2^{p-1}|x|^p+2^{p-1}|y|^p$ we have 
 \begin{equation*}
     \begin{split}
         \int_{e_1}\int_{e_2} \frac{|\varphi(x)-\varphi(y)|^p}{d(x,y)^{1+\alpha p}}dxdy&<C\int_{e_1}|\varphi(x)|^pdx+C\int_{e_2}|\varphi(x)|^pdx\\
         &= C \left(\|\varphi\|_{\mathbb{L}_p(e_1)}^p + \|\varphi\|_{\mathbb{L}_p(e_2)}^p\right).
     \end{split}
 \end{equation*}
 Now suppose $v\in e_1\cap e_2$ for $v\in\mathcal{V}$, let $\zeta_1,\zeta_2$ be natural parameterizations of $e_1$ and $e_2$ that satisfy $\zeta_1(0)=\zeta_2(0)=v$. Let $\epsilon\in(0,\frac{1}{2}\min_{e\in\mathcal{E}}l_e)$ such that 
 \begin{equation}\label{eq:distance_1}
     \begin{split}
         \min\Big\{d(\zeta_1(t),\zeta_2(s)), d(\zeta_1(l_{e_1}-t),\zeta_2(l_{e_2}-s)), &\\d(\zeta_1(l_{e_1}-t),\zeta_2(s)), d(\zeta_1(t),\zeta_2(l_{e_2}-s))\Big\}  &=|t+s|
     \end{split}
 \end{equation}
for every $t,s\in(0,\epsilon)$, and
 \begin{equation}\label{eq:distance_2}
         \inf_{(t,s)\in S}d(\zeta_1(t),\zeta_2(s))>r.
 \end{equation}
for a real number $r>0$, where $$S=(0,l_{e_1})\times(0,l_{e_2})\backslash\left(\Big((0,\epsilon)\cup(l_{e_1}-\epsilon,l_{e_1})\Big)\times \Big((0,\epsilon)\cup(l_{e_2}-\epsilon,l_{e_2})\right)\Big).$$

Let $\zeta:[-l_{e_1},l_{e_2}]\to e_1\cup e_2$ defined by
\begin{equation*}
    \zeta(t) = \begin{cases}
        \zeta_1(-t),\quad t\in[-l_{e_1},0]\\
        \zeta_2(t),\quad t\in[0,l_{e_2}],
    \end{cases}
\end{equation*}
and set $\tilde{\varphi}=\varphi\circ\zeta$. We have
 \begin{equation}\label{eq:integral_1}
     \begin{split}
        &\mathrel{\hphantom{\leq}} \int_{e_1}\int_{e_2} \frac{|\varphi(x)-\varphi(y)|^p}{d(x,y)^{1+\alpha p}}dxdy
        =\int_{-l_{e_1}}^0\int_0^{l_{e_2}} \frac{|\tilde{\varphi}(x)-\tilde{\varphi}(y)|^p}{d(\zeta(x),\zeta(y))^{1+\alpha p}}dxdy\\
        &\leq r^{-1-\alpha p}\iint_S|\tilde{\varphi}(x)-\tilde{\varphi}(y)|^pdxdy + \sum_{i=1}^4\int_{-\epsilon}^{0}\int_0^{\epsilon}\frac{|\tilde{\varphi}_i(x)-\tilde{\varphi}_i(y)|^p}{|x-y|^{1+\alpha p}}dxdy
     \end{split}
 \end{equation}
where the four functions $\tilde{\varphi_i}$ are defined as
\begin{align*}
        \tilde{\varphi}_1(x) &= \tilde{\varphi}(x),&\tilde{\varphi}_2(x) &= \begin{cases}
            \tilde{\varphi}(x),\;\; x\in[0,\epsilon],\\
            \tilde{\varphi}(-l_{e_2}-x),\;\; x\in[-\epsilon,0],\\
        \end{cases}\\
                \tilde{\varphi}_3(x) &= \begin{cases}
            \tilde{\varphi}(l_{e_1}-x),\;\; x\in[0,\epsilon],\\
            \tilde{\varphi}(-l_{e_2}+x),\;\; x\in[-\epsilon,0],\\
        \end{cases}& \tilde{\varphi}_4(x) &= \begin{cases}
            \tilde{\varphi}(l_{e_1}-x),\;\; x\in[0,\epsilon],\\
            \tilde{\varphi}(x),\;\; x\in[-\epsilon,0].\\
        \end{cases}
\end{align*}

The first term appearing in the last line of the estimate \eqref{eq:integral_1} is bounded above by $C\left(\|\varphi\|_{\mathbb{L}_p(e_1)}^p+\|\varphi\|_{\mathbb{L}_p(e_2)}^p\right)$. To handle the sum we use the inequality $|x-y|^p\leq 2^{p-1}|x|^p+2^{p-1}|y|^p$ and Fubini's Theorem as follows
\begin{equation}\label{eq:integral_2}
    \begin{split}
        &\mathrel{\hphantom{\leq}}\int_{-\epsilon}^{0}\int_0^{\epsilon}\frac{|\tilde{\varphi}_i(x)-\tilde{\varphi}_i(y)|^p}{|x-y|^{1+\alpha p}}dxdy\\
        &\leq C\int_0^{\epsilon}\int_{-\epsilon}^0\frac{|\tilde{\varphi}_i(x)|^p}{|x-y|^{1+\alpha p}}dydx + C\int_{-\epsilon}^0\int_0^{\epsilon}\frac{|\tilde{\varphi}_i(y)|^p}{|x-y|^{1+\alpha p}}dxdy\\
        & = C\frac{1}{\alpha}\int_0^{\epsilon}|\tilde{\varphi}_i(x)|^px^{-\alpha p}dx - \frac{C}{\alpha}\int_0^{\epsilon}|\tilde{\varphi}_i(x)|^p(x+\epsilon)^{-\alpha p}dx \\&- \frac{C}{\alpha}\int_{-\epsilon}^0|\tilde{\varphi}_i(y)|^p(\epsilon-y)^{-\alpha p}dy + \frac{C}{\alpha}\int_{-\epsilon}^0|\tilde{\varphi}_i(y)|^p(-y)^{-\alpha p}dy\\
        &\leq\frac{C}{\alpha}\left(\||x|^{-\alpha}\tilde{\varphi}_i\|_{\mathbb{L}_p(0,\epsilon)}^p + \||x|^{-\alpha}\tilde{\varphi}_i\|_{\mathbb{L}_p(-\epsilon,0)}^p\right)\\ &+ \frac{C\epsilon^{-\alpha p}}{\alpha}\left(\|\tilde{\varphi}_i\|_{\mathbb{L}_p(-\epsilon,0)}^p + \|\tilde{\varphi}_i\|_{\mathbb{L}_p(0,\epsilon)}^p\right).
    \end{split}
\end{equation}
 
Therefore, it is enough to show that 
\begin{equation}\label{eq:x_alpha_phi}
\begin{split}
    \||x|^{-\alpha}\tilde{\varphi}_i\|_{\mathbb{L}_p(-\epsilon,0)} &\leq C \|\tilde{\varphi}_i\|_{W^{\alpha,p}(-\epsilon,0)},\\\||x|^{-\alpha}\tilde{\varphi}_i\|_{\mathbb{L}_p(0,\epsilon)} &\leq C \|\tilde{\varphi}_i\|_{W^{\alpha,p}(0,\epsilon)},
\end{split}
\end{equation}
for a constant $C$ independent of $\tilde{\varphi}_i$. By \citeA[Theorem~5.4]{di_nezza} let $\phi_{i,1},\phi_{i,2}\in W^{\alpha,p}(\mathbb{R})$ be extensions of $\tilde{\varphi}_i|_{(0,\epsilon)}$ and $\tilde{\varphi}_i|_{(-\epsilon,0)}$ respectively to $\mathbb{R}$ that satisfy
$$\|\phi_{i,1}\|_{W^{\alpha,p}(\mathbb{R})}\leq C\|\tilde{\varphi}_i\|_{W^{\alpha,p}(-\epsilon,0)},\quad\|\phi_{i,2}\|_{W^{\alpha,p}(\mathbb{R})}\leq C\|\tilde{\varphi}_i\|_{W^{\alpha,p}(0,\epsilon)},$$
for a constant $C$ independent of $\varphi$. By \citeA[Theorem~2.4]{di_nezza} we can find two sequences $\{\psi_n\}$ and $\{\xi_n\}$ of functions in $C^{\infty}_c(\mathbb{R})$ that satisfy $$\|\phi_{i,1}-\psi_n\|_{W^{\alpha,p}(\mathbb{R})}<\frac{1}{n},\text{ and } \|\phi_{i,2}-\xi_n\|_{W^{\alpha,p}(\mathbb{R})}<\frac{1}{n}.$$ Given this construction and using \citeA[Theorem~1.1]{nguyen} we have 
 \begin{equation}\label{eq:psi_W_1}
 \begin{split}
     \||x|^{-\alpha}\psi_n\|_{\mathbb{L}_p(\mathbb{R})}\leq C\|\psi_n\|_{W^{\alpha,p}(\mathbb{R})}
     &\leq \frac{C}{n} + C\|\phi_{i,1}\|_{W^{\alpha,p}(0,\epsilon)}\\
     &\leq \frac{C}{n} + C\|\tilde{\varphi}_i\|_{W^{\alpha,p}(0,\epsilon)}
 \end{split}
 \end{equation}
 and a similar inequality holds for $\xi_n$. We calculate using the monotone convergence theorem
\begin{equation}\label{eq:x_alpha_Lp_1}
    \begin{split}
        &\mathrel{\hphantom{=}}\||x|^{-\alpha}\tilde{\varphi}_i\|_{\mathbb{L}_p(0,\epsilon)} + \||x|^{-\alpha}\tilde{\varphi}_i\|_{\mathbb{L}_p(-\epsilon,0)}\\
        &= \lim_{n\to\infty}\||x|^{-\alpha}\tilde{\varphi}_i\|_{\mathbb{L}_p(\frac{1}{n},\epsilon)} + \lim_{n\to\infty}\||x|^{-\alpha}\tilde{\varphi}_i\|_{\mathbb{L}_p(-\epsilon,-\frac{1}{n})}\\
        &\leq \underset{n\to\infty}{\lim \sup}\||x|^{-\alpha}(\tilde{\varphi}_i-\psi_n)\|_{\mathbb{L}_p(\frac{1}{n},\epsilon)} +\underset{n\to\infty}{\lim \sup} \||x|^{-\alpha}\psi_n\|_{\mathbb{L}_p(\mathbb{R})}\\
        &+ \underset{n\to\infty}{\lim \sup}\||x|^{-\alpha}(\tilde{\varphi}_i-\xi_n)\|_{\mathbb{L}_p(-\epsilon,-\frac{1}{n})} + \underset{n\to\infty}{\lim \sup} \||x|^{-\alpha}\xi_n\|_{\mathbb{L}_p(\mathbb{R})}.
    \end{split}
\end{equation}
By \eqref{eq:psi_W_1}, both terms $\underset{n\to\infty}{\lim \sup} \||x|^{-\alpha}\psi_n\|_{\mathbb{L}_p(\mathbb{R})}$ and $\underset{n\to\infty}{\lim \sup} \||x|^{-\alpha}\xi_n\|_{\mathbb{L}_p(\mathbb{R})}$ are bounded from above by $C(\|\tilde{\varphi}_i\|_{W^{\alpha,p}(0,\epsilon)}+\|\tilde{\varphi}_i\|_{W^{\alpha,p}(-\epsilon,0)})$. For the term $\underset{n\to\infty}{\lim \sup}\||x|^{-\alpha}(\tilde{\varphi}_i-\psi_n)\|_{\mathbb{L}_p(\frac{1}{n},\epsilon)}$, we let $p^*=\frac{p}{1-\alpha p}$ and use \citeA[Theorem~6.7]{di_nezza} along with H\"older's inequality to obtain
 \begin{equation}\label{eq:x_alpha_Lp_2}
     \begin{split}
         \||x|^{-\alpha}(\tilde{\varphi}_i-\psi_n)\|_{\mathbb{L}_p(\frac{1}{n},1)}&\leq \|\tilde{\varphi}_i-\psi_n\|_{\mathbb{L}_{p^*}(0,\epsilon)}^p\||x|^{-1}\|^{\alpha p}_{\mathbb{L}_{1}(\frac{1}{n},\epsilon)}\\
         &\leq C\|\tilde{\varphi}_i-\psi_n\|_{W^{\alpha,p}(0,\epsilon)}^p\||x|^{-1}\|^{\alpha p}_{\mathbb{L}_{1}(\frac{1}{n},\epsilon)}\\
         &\leq C\frac{1}{n^p}|\log(n)|^{\alpha p}.
     \end{split}
 \end{equation}
The last quantity approaches $0$ as $n\to\infty$. Similar calculations hold for $\underset{n\to\infty}{\lim \sup}\||x|^{-\alpha}(\tilde{\varphi}_i-\xi_n)\|_{\mathbb{L}_p(-\epsilon,-\frac{1}{n})}$. This shows \eqref{eq:x_alpha_phi}, and hence yields the desired result.
\end{proof}
Theorem~\ref{thm:W_characterization_1} shows that for low smoothness ($0<\alpha<\frac{1}{p}$), continuity at vertices is not enforced. In contrast, once $\alpha>\frac{1}{p}$ continuity across vertices is enforced as shown next.

\begin{theorem}\label{thm:W_characterization_2}
        Let $p\in(1,\infty)$ and $\alpha\in\left(\frac{1}{p},1\right)$. We have $$W^{\alpha,p}(\Gamma)\cong\bigoplus_{e\in\mathcal{E}}W^{\alpha,p}(e)\cap C(\Gamma).$$
\end{theorem}
\begin{proof}
We first show the inclusion $\bigoplus_{e\in\mathcal{E}}W^{\alpha,p}(e)\cap C(\Gamma)\hookrightarrow W^{\alpha,p}(\Gamma)$. We follow the steps of the proof of Theorem \ref{thm:W_characterization_1}, thus we must bound integrals of the form in \eqref{eq:integral_1}. Let $\varphi\in \bigoplus_{e\in\mathcal{E}}W^{\alpha,p}(e)\cap C(\Gamma)$, Let $\zeta$ be as in the proof of Theorem \ref{thm:W_characterization_1} and let $\tilde{\varphi}=\varphi\circ\zeta$. Define the sets
\begin{equation*}
    \begin{split}
        S_1 &= (-\epsilon,0)\times(0,\epsilon);\\
        S_2 &= \begin{cases}
    (-\epsilon,0)\times (l_{e_2}-\epsilon,l_{e_2}),\;if\; \zeta(0) = \zeta(l_{e_2}), \\
    \varnothing, \; otherwise;
\end{cases}\\
        S_3 &= \begin{cases}
    (-l_{e_1},-l_{e_1}+\epsilon)\times (l_{e_2}-\epsilon,l_{e_2}),\;if\; \zeta(l_{e_2}) = \zeta(-l_{e_1}), \\
    \varnothing, \; otherwise;
\end{cases}\\
    S_4 &= \begin{cases}
    (-l_{e_1},-l_{e_1}+\epsilon)\times (0,\epsilon),\;if\; \zeta(0) = \zeta(l_{e_1}), \\
    \varnothing, \; otherwise;
\end{cases}
    \end{split}
\end{equation*}
and $S = (-l_{e_1},0)\times(0,l_{e_2})\backslash \bigcup_{i=1}^4 S_i.$
We have that $$\inf_{(s,t)\in S}d(\zeta(s),\zeta(t))> r >0$$ for a positive real number $r$. Define $\tilde{\varphi}_i$ as in the proof of Theorem \ref{thm:W_characterization_1}. Similarly to \eqref{eq:integral_1} we have
 \begin{equation*}
     \begin{split}
        &\mathrel{\hphantom{\leq}} \int_{e_1}\int_{e_2} \frac{|\varphi(x)-\varphi(y)|^p}{d(x,y)^{1+\alpha p}}dxdy\\
        &\leq r^{-1-\alpha p}\iint_S|\tilde{\varphi}(x)-\tilde{\varphi}(y)|^pdxdy + \sum_{\substack{i\in\{1,2,3,4\}\\S_i\not=\varnothing}}^4\int_{-\epsilon}^{0}\int_0^{\epsilon}\frac{|\tilde{\varphi}_i(x)-\tilde{\varphi}_i(y)|^p}{|x-y|^{1+\alpha p}}dxdy.
     \end{split}
 \end{equation*}
The first term is bounded by $C\left(\|\varphi\|_{\mathbb{L}_p(e_1)}^p+\|\varphi\|_{\mathbb{L}_p(e_2)}^p\right)$. For each $i$ that satisfies $S_i\not=\varnothing$, we argue as in \eqref{eq:integral_2}:
\begin{equation*}
\begin{split}
    &\mathrel{\hphantom{\leq}}\int_{-\epsilon}^{0}\int_0^{\epsilon}\frac{|\tilde{\varphi}_i(x)-\tilde{\varphi}_i(y)|^p}{|x-y|^{1+\alpha p}}dxd\\
    &\leq C \int_{-\epsilon}^{0}\int_0^{\epsilon}\frac{|\tilde{\varphi}_i(x)-\tilde{\varphi}_i(0)|^p}{|x-y|^{1+\alpha p}}dxd + C \int_{-\epsilon}^{0}\int_0^{\epsilon}\frac{|\tilde{\varphi}_i(0)-\tilde{\varphi}_i(y)|^p}{|x-y|^{1+\alpha p}}dxd\\
    &\leq\frac{C}{\alpha}\||x|^{-\alpha}(\tilde{\varphi}_i-\tilde{\varphi}_i(0))\|_{\mathbb{L}_p(-\epsilon,\epsilon)}^p + \frac{C}{\alpha}\|\tilde{\varphi}_i-\tilde{\varphi}_i(0)\|_{\mathbb{L}_p(-\epsilon,\epsilon)}^p.
\end{split}
\end{equation*}
Let $\phi:\mathbb{R}\to\mathbb{R}$ be a smooth function that satisfies $\phi([-\frac{1}{2},\frac{1}{2}])=\{0\}$ and $\phi(\mathbb{R}\backslash(-1,1))=\{1\}$. For every $\delta>0$ define $\phi_{\delta}$ by $\phi_{\delta}(x) = \phi(\frac{x}{\delta})$. Define $\bar{\varphi}_{\delta} = (\tilde{\varphi}_i-\tilde{\varphi}_i(0))\phi_\delta$, the function $\bar{\varphi}_\delta$ converges to $\tilde{\varphi}_i$ everywhere as $\delta\to0^+$. Moreover, because $\tilde{\varphi}_i$ is continuous hence bounded the dominated convergence theorem implies that $\bar{\varphi}_\delta$ converges to $\tilde{\varphi}_i$ in $\mathbb{L}_p(-\epsilon,\epsilon)$. Dominated convergence further imply the convergence of $\bar{\varphi}_\delta$ to $\tilde{\varphi}_i$ in $W^{\alpha,p}(0,\epsilon)$ and in $W^{\alpha,p}(-\epsilon,0)$. This follows because the function $$\left(|\tilde{\varphi}_i(y)-\tilde{\varphi}_i(x)-\bar{\varphi}_\delta(x)+\bar{\varphi}_\delta(y)|^p\right)/|x-y|^{1+\alpha p}$$ is dominated by the integrable function $C\left(|\tilde{\varphi}_i(y)-\tilde{\varphi}_i(x)|^p\right)/|x-y|^{1+\alpha p}$. For each $n\in\mathbb{N}$, choose $\delta_n>0$ such that $$\|\tilde{\varphi}_i-\tilde{\varphi}_i(0)-\bar{\varphi}_{\delta_n}\|_{W^{\alpha,p}(-\epsilon,0)}+\|\tilde{\varphi}_i-\tilde{\varphi}_i(0)-\bar{\varphi}_{\delta_n}\|_{W^{\alpha,p}(0,\epsilon)}<\frac{1}{n}.$$ As in the proof of Theorem \ref{thm:W_characterization_1} let $\chi_n$ be an extension of $\bar{\varphi}_{\delta_n}$ to $\mathbb{R}$ that satisfies $$\|\chi_n\|_{W^{\alpha,p}(\mathbb{R})}\leq C\min\{\|\bar{\varphi}_{\delta_n}\|_{W^{\alpha,p}(0,\epsilon)},\|\bar{\varphi}_{\delta_n}\|_{W^{\alpha,p}(-\epsilon,0)}\}$$ for a constant $C$ independent of $\bar{\varphi}_{\delta_n}$. Let $\{\psi_n\}$ be
a $C^{\infty}_c(\mathbb{R}\backslash\{0\})$ that satisfies
$$\|\chi_n - \psi_n\|_{W^{\alpha,p}(\mathbb{R})}\leq \frac{1}{n}.$$
Following calculations similar to \eqref{eq:x_alpha_Lp_1} and \eqref{eq:x_alpha_Lp_2} and invoking the second part of \citeA[Theorem~1.1]{nguyen} establishes the desired inclusion.\\
For the inclusion $W^{\alpha,p}(\Gamma)\hookrightarrow \bigoplus_{e\in\mathcal{E}}W^{\alpha,p}(e)\cap C(\Gamma)$ it suffices to show that $W^{\alpha,p}(\Gamma)\subset C(\Gamma)$. By \citeA[Theorem~8.2]{di_nezza} we have $W^{\alpha,p}(e)\hookrightarrow C(e)$ for every $e\in\mathcal{E}$. It remains to verify that the functions of $W^{\alpha,p}(\Gamma)$ are continuous at the vertices of $\Gamma$. Let $\varphi\in W^{\alpha,p}(\Gamma)$ and let $e_1,e_2\in\mathcal{E}$ be two distinct edges such that $v\in e_1\cap e_2$. Parameterize $e_1\cup e_1$ by the interval $[-l_{e_1},l_{e_2}]$ so that $0$ corresponds to $v$ as done above. Let $a=\lim_{x\to0^+}\tilde{\varphi}(x)$ and $b=\lim_{x\to0^-}\tilde{\varphi}(x)$ and suppose, seeking a contradiction, that $a\not=b$. By continuity of $\varphi$ on $e_1$ and $e_2$ there exists $\epsilon>0$ that satisfies $\sup_{0<x<\epsilon}|\tilde{\varphi}(x)-a| <\frac{|b-a|}{4}$ and $\sup_{-\epsilon<x<0}|\tilde{\varphi}(x)-b| <\frac{|b-a|}{4}$, hence $|\tilde{\varphi}(x)-\tilde{\varphi}(y)|>\frac{|b-a|}{2}$ for every $(x,y)\in (-\epsilon,0)\times(0,\epsilon)$, consequently
\begin{equation*}
    \int_{-l_{e_1}}^0\int_0^{l_{e_2}}\frac{|\tilde{\varphi}(x)-\tilde{\varphi}(y)|^p}{|x-y|^{1+\alpha p}}\geq \left(\frac{|b-a|}{2}\right)^p\int_{-\epsilon}^0\int_0^{\epsilon}\frac{1}{|x-y|^{1+\alpha p}} = \infty
\end{equation*}
which contradicts that $\varphi\in W^{\alpha,p}(\Gamma)$, therefore $a=b$ and $\varphi$ is continuous.
\end{proof}
Next, we extend the characterizations in Theorems~\ref{thm:W_characterization_1} and~\ref{thm:W_characterization_2} to larger values of $\alpha$. The following general characterization reveals an alternating regularity pattern: even-order derivatives up to a critical order must be continuous across vertices, while odd-order derivatives remain unrestricted.

\begin{theorem}\label{thm:W_chatacterization_3}
Let $p> 1$ and $\alpha\in(0,\infty)\backslash\{2k+\frac{1}{p}: k\in\mathbb{N}\}$. We have
        \begin{equation}\label{eq:W_char_main_eq}
            W^{\alpha,p}(\Gamma)\cong \bigoplus_{e\in\mathcal{E}}W^{\alpha,p}(e) \cap C^{\left(2\floor{\frac{\alpha}{2}-\frac{1}{2p}}\right)}(\Gamma)
        \end{equation}
        where we use the notation $C^{(-2)}(\Gamma)=\mathbb{L}_2(\Gamma)$. 
\end{theorem}
\begin{proof}
    For $\alpha\in(0,1)$ the result follows from Theorem~\ref{thm:W_characterization_1} and Theorem~\ref{thm:W_characterization_2}. For $\alpha\in\mathbb{N}\cup[1,2)$ the definition of $W^{\alpha,p}(\Gamma)$ is precisely \eqref{eq:W_char_main_eq}. For $\alpha\in\left[2,2+\frac{1}{p}\right)$, building upon the previously mentioned results, Definition~\ref{def:W_Gamma_1}, and the fact that $D:W^{\alpha,p}(e)\to W^{\alpha-1,p}(e)$ is a bounded operator (\citeA[Proposition~2.1]{di_nezza}) we have that $\varphi\in W^{\alpha,p}(\Gamma)$ if and only if 
    \begin{equation*}
        \begin{split}
            \varphi&\in\bigoplus_{e\in\mathcal{E}}W^{\alpha,p}(e)\cap W^{2,p}(\Gamma)\cong \bigoplus_{e\in\mathcal{E}}W^{\alpha,p}(e)\cap C(\Gamma),\\
            D^2\varphi &\in W^{\alpha-2,p}(\Gamma)\cong \bigoplus_{e\in\mathcal{E}}W^{\alpha-2,p}(e).
        \end{split}
    \end{equation*}
    Combining this with 
    $$\|\varphi\|_{W^{\alpha,p}(\Gamma)}^p = \|\varphi\|_{W^{\floor{\alpha},p}(\Gamma)}^p + \|D^{\alpha}\varphi\|_{W^{\alpha-\floor{\alpha}}(\Gamma)}^p,$$
    proves \eqref{eq:W_char_main_eq} for this case.
    We proceed inductively on $n\in\mathbb{N}_0$ by showing that \eqref{eq:W_char_main_eq} holds for every $\alpha\in\left(2n+\frac{1}{p},2n+2+\frac{1}{p}\right)$.
    Supposing that \eqref{eq:W_char_main_eq} holds for every $\alpha\in\left(2n+\frac{1}{p},2n+2+\frac{1}{p}\right)$, we will show that it holds for every $\alpha\in\left(2n+2+\frac{1}{p},2n+4+\frac{1}{p}\right)$ by considering several cases. If $\alpha\in\left(2n+2+\frac{1}{p},2n+3\right)$ then by Definition~\ref{def:W_Gamma_1} and the fact that \eqref{eq:W_char_main_eq} holds for integers we have that $\varphi\in W^{\alpha,p}(\Gamma)$ if and only if
    \begin{equation*}
        \begin{split}
            \varphi&\in\bigoplus_{e\in\mathcal{E}}W^{\alpha,p}(e)\cap W^{2n+2,p}(\Gamma)\cong\bigoplus_{e\in\mathcal{E}}W^{\alpha,p}(e)\cap C^{(2n)}(\Gamma),
        \end{split}
    \end{equation*}
    and $D^{2n+2}\varphi\in W^{\alpha-\floor{\alpha},p}(\Gamma)$. Because $\alpha-\floor{\alpha} \in\left(\frac{1}{p},1\right)$ we apply Theorem~\ref{thm:W_characterization_2} and \citeA[Proposition~2.1]{di_nezza}, to find that $D^{2n+2}\varphi$ is continuous, hence $$W^{\alpha,p}(\Gamma)\cong \bigoplus_{e\in\mathcal{E}}W^{\alpha,p}(e)\cap C^{(2n+2)}(\Gamma).$$ If $\alpha\in[2n+3,2n+4)$, then immediately from Definition~\ref{def:W_Gamma_1} we have that $W^{\alpha,p}(\Gamma) = \bigoplus_{e\in\mathcal{E}}W^{\alpha,p}(e)\cap C^{(2n+2)}(\Gamma)$. If $\alpha\in\left[2n+3,2n+4+\frac{1}{p}\right)$, then again employing that $\eqref{eq:W_char_main_eq}$ holds for integers we have that $\varphi\in W^{\alpha,p}(\Gamma)$ if and only if
    $$\varphi\in\bigoplus_{e\in\mathcal{E}}W^{\alpha,p}(e)\cap C^{(2n+2)}(\Gamma) \textrm{ for which } D^{2n+2}\varphi\in W^{\alpha-\floor{\alpha},p}(\Gamma),$$
    because $\alpha-\floor{\alpha}\in\left(0,\frac{1}{p}\right)$. Then Theorem \ref{thm:W_characterization_2} and \citeA[Proposition~2.1]{di_nezza} yields 
    $$W^{\alpha,p}(\Gamma)\cong \bigoplus_{e\in\mathcal{E}}W^{\alpha,p}(e)\cap C^{(2n+2)}(\Gamma).$$
    This finishes the induction and proves the theorem.
\end{proof}

Next, we extend the classical Sobolev embedding theorems to the metric graph setting.

\begin{theorem}\label{thm:sobolev_subspace}
    Let $p\in(1,\infty)$ and $0<\alpha<\beta<\infty$. Then $W^{\beta,p}(\Gamma)\hookrightarrow W^{\alpha,p}(\Gamma)$.
\end{theorem}

\begin{proof} The inclusion $W^{\beta,p}(\Gamma)\subset W^{\alpha,p}(\Gamma)$ is clear. We now show the norm relation in several cases depending on $\alpha$ and $\beta$.\\
\noindent\textbf{Case 1:} $\beta<1$.
    Let $\mathcal{S}=\{(x,y)\in\Gamma\times\Gamma : d(x,y)\leq 1\}$, using the inequality $|a-b|^p\leq 2^{p-1}|a|^p+2^{p-1}|b|^p$ we have
    \begin{equation*}
        \begin{split}
            &\quad\int_{\Gamma}\int_{\Gamma} \frac{|f(x)-f(y)|^p}{d(x,y)^{1+\alpha p}}dxdy \\&\leq \int_{\mathcal{S}} \frac{|f(x)-f(y)|^p}{d(x,y)^{1+\alpha p}}dxdy  +2^{p-1}\int_{\Gamma^2}(|f(x)|^p + |f(y)|^p)dxdy\\
            &\leq \int_{\mathcal{S}} \frac{|f(x)-f(y)|^p}{d(x,y)^{1+\beta p}}dxdy + C\|f\|_{\mathbb{L}_p(\Gamma)}^p\leq C\|f\|_{W^{\beta,p}(\Gamma)}^p.
        \end{split}
    \end{equation*}
Combining this inequality with the inequality $\|f\|_{W^{\floor{\alpha},p}(\Gamma)}\leq \|f\|_{W^{\beta,p}(\Gamma)}$ we have $\|f\|_{W^{\alpha,p}(\Gamma)}\leq \|f\|_{W^{\beta,p}(\Gamma)}$. Therefore $W^{\beta,p}(\Gamma)\hookrightarrow W^{\alpha,p}(\Gamma)$.\\

\noindent\textbf{Case 2:} $\beta=1$, $\alpha\not=\frac{1}{p}$.
By Theorem~\ref{thm:W_chatacterization_3} we have $W^{1,p}(\Gamma) = \bigoplus_{e\in\mathcal{E}}W^{1,p}(e)\cap C(\Gamma)$ and, depending on the value of $\alpha$, either $W^{\alpha,p}(\Gamma) \cong \bigoplus_{e\in\mathcal{E}}W^{\alpha,p}(e)\cap C(\Gamma)$ or $W^{\alpha,p}(\Gamma) \cong \bigoplus_{e\in\mathcal{E}}W^{\alpha,p}(e)$. \citeA[Proposition~2.2]{di_nezza} states that $W^{1,p}(e)\hookrightarrow W^{\alpha,p}(e)$, therefore, in both cases of $\alpha$ we have $W^{1,p}(\Gamma)\hookrightarrow W^{\alpha,p}(\Gamma)$.\\

\noindent\textbf{Case 3:} $\beta=1$, $\alpha=\frac{1}{p}$. By Case~1 we have $W^{\frac{p+1}{2p},p}(\Gamma)\hookrightarrow W^{\alpha,p}(\Gamma)$, and by Case~2 we have $W^{1,p}(\Gamma)\hookrightarrow W^{\frac{p+1}{2p},p}(\Gamma)$. Therefore $W^{1,p}(\Gamma)\hookrightarrow W^{\alpha,p}(\Gamma)$.\\

\noindent\textbf{Case 4:} $\floor{\alpha}=\floor{\beta}$. Follows analogously to Case~1.\\

\noindent\textbf{Case 5:} $\floor{\alpha}<\floor{\beta}$. By Case~4 we have $W^{\beta,p}(\Gamma)\hookrightarrow W^{\floor{\beta},p}(\Gamma)$, and by definition of integer indexed Sobolev spaces we have $W^{\floor{\beta},p}(\Gamma)\hookrightarrow W^{\floor{\alpha}+1,p}(\Gamma)$. Finally $ W^{\floor{\alpha}+1,p}(\Gamma) \hookrightarrow W^{\alpha,p}(\Gamma)$ follows from definition and an argument similar to Case~2 and Case~3.
\end{proof}

\begin{theorem}\label{thm:embedding_1}
    Let $(p,\alpha)\in[1,\infty)\times(0,1)$.
    \begin{enumerate}[i)]
        \item if $\alpha<1/p$ then $W^{\alpha,p}(\Gamma)\compactright \mathbb{L}_k(\Gamma)$ for any $k<p/(1-\alpha p)$.
        \item if $\alpha=1/p$ then $W^{\alpha,p}(\Gamma)\compactright \mathbb{L}_k(\Gamma)$ for any $k<\infty$.
        \item if $\alpha>1/p$ then $W^{\alpha,p}(\Gamma)\compactright C^{(0),\lambda}(\Gamma)$ for any $\lambda<\alpha - 1/p$.
    \end{enumerate}
\end{theorem}
\begin{proof}
    The continuous inclusion $W^{\alpha,p}(\Gamma)\hookrightarrow \bigoplus_{e\in\mathcal{E}} W^{\alpha,p}(e)$ is immediate from definition. By \citeA[Theorem~4.58]{demengel} we have $$\bigoplus_{e\in\mathcal{E}} W^{\alpha,p}(e)\compactright \bigoplus_{e\in\mathcal{E}} \mathbb{L}_k(e) = \mathbb{L}_k(\Gamma)$$ for any $k<p/(1-\alpha p)$, hence the first statement follows. The second statement follows similarly by \citeA[Theorem~4.58]{demengel}. The third statement follows similarly by employing Theorem~\ref{thm:W_chatacterization_3} and \citeA[Theorem~4.58]{demengel}:
    $$W^{\alpha,p}(\Gamma)\hookrightarrow\bigoplus_{e\in\mathcal{E}}W^{\alpha,p}(e)\cap C(\Gamma)\compactright\bigoplus_{e\in\mathbb{E}}C^{0,\lambda}(e)\cap C(\Gamma) = C^{(0),\lambda}(\Gamma).$$
\end{proof}
The following theorem shows a generalization of the third statement of Theorem~\ref{thm:embedding_1}.

\begin{theorem}\label{thm:W_imbedding_1}
    Let $p>1$ and $\alpha\in(\frac{1}{p},\infty)\backslash\{2k+\frac{1}{p} : k\in\mathbb{N}\}$. We have the following compact inclusions.
    \begin{enumerate}[i)]
        \item If $\alpha\in\bigcup_{k\in\mathbb{N}_0}\big(2k+\frac{1}{p},2k+1+\frac{1}{p}\big)$ then $W^{\alpha,p}(\Gamma)\compactright C^{\left(\floor{\alpha-\frac{1}{p}}\right), \lambda}(\Gamma)$ for any non-negative $\lambda<\alpha-\frac{1}{p} - \floor{\alpha-\frac{1}{p}}$.
        \item If $\alpha\in\bigcup_{k\in\mathbb{N}_0}\big[2k+1+\frac{1}{p},2k+2+\frac{1}{p}\big)$ then $W^{\alpha,p}(\Gamma)\compactright C^{\left(2\floor{\frac{\alpha}{2}-\frac{1}{2p}}\right), \lambda}(\Gamma)$ for any $\lambda\leq1$.
    \end{enumerate}
\end{theorem}
\begin{proof}
    Suppose that $\alpha\in\bigcup_{k\in\mathbb{N}}\big(2k+\frac{1}{p},2k+1+\frac{1}{p}\big)$. In this case we have $\floor{\alpha-\frac{1}{p}}=2\floor{\frac{\alpha}{2}-\frac{1}{2p}}$. For any $0\leq\lambda<\alpha-\frac{1}{p} - \floor{\alpha-\frac{1}{p}}$ and $e\in\mathcal{E}$ we have by \citeA[Theorem~4.58]{demengel} that $W^{\alpha,p}(e)\compactright C^{\floor{\alpha-\frac{1}{p}},\lambda}(e)$. Using Theorem~\ref{thm:W_chatacterization_3} we have
    \begin{equation*}
        \begin{split}
            W^{\alpha,p}(\Gamma)&\cong\bigoplus_{e\in\mathcal{E}}W^{\alpha,p}(e)\cap C^{\left(2\floor{\frac{\alpha}{2}-\frac{1}{2p}}\right)}(\Gamma)\\&\compactright\bigoplus_{e\in\mathcal{E}}C^{\floor{\alpha-\frac{1}{p}},\lambda}(e)\cap C^{\left(\floor{\alpha-\frac{1}{p}}\right)}(\Gamma) \cong C^{\left(\floor{\alpha-\frac{1}{p}}\right),\lambda}(\Gamma).
        \end{split}
    \end{equation*}
    This shows the first statement. For the second statement we have  $\floor{\alpha-\frac{1}{p}}=2\floor{\frac{\alpha}{2}-\frac{1}{2p}}+1$. Let $\rho = \frac{1}{2}\left(\alpha-\frac{1}{p} - \floor{\alpha-\frac{1}{p}}\right)$ and $\lambda\in[0,1]$. By \citeA[Theorem~4.58]{demengel} we have
    \begin{equation*}
        \begin{split}
            W^{\alpha,p}(\Gamma)&\cong\bigoplus_{e\in\mathcal{E}}W^{\alpha,p}(e)\cap C^{\left(2\floor{\frac{\alpha}{2}-\frac{1}{2p}}\right)}(\Gamma)\\&\compactright\bigoplus_{e\in\mathcal{E}}C^{2\floor{\frac{\alpha}{2}-\frac{1}{2p}}+1,\rho}(e)\cap C^{\left(2\floor{\frac{\alpha}{2}-\frac{1}{2p}}\right)}(\Gamma) \\
            &\hookrightarrow C^{\left(2\floor{\frac{\alpha}{2}-\frac{1}{2p}}\right),1}(\Gamma)\hookrightarrow C^{\left(2\floor{\frac{\alpha}{2}-\frac{1}{2p}}\right),\lambda}(\Gamma).
        \end{split}
    \end{equation*}
    The second continuous inclusion follows because by means of the mean value theorem and \eqref{eq:lip_proof}, the functions in $$\bigoplus_{e\in\mathcal{E}}C^{2\floor{\frac{\alpha}{2}-\frac{1}{2p}}+1,\rho}(e)\cap C^{\left(2\floor{\frac{\alpha}{2}-\frac{1}{2p}}\right)}(\Gamma)$$ have a $2\floor{\frac{\alpha}{2}-\frac{1}{2p}}$ derivative that is Lipschitz. The third continuous inclusion is a consequence of the compactness of $\Gamma$.
\end{proof}
A classical result in analysis in the fact that $W^{1,\infty}$ is the space of Lipschitz functions; this result extends to the metric graphs scenario as we show below.
\begin{theorem}
    $W^{1,\infty}(\Gamma)$ is the space of Lipschitz functions on $\Gamma$.
\end{theorem}
\begin{proof}
    Let $\varphi\in W^{1,\infty}(\Gamma)$, by \citeA[Theorem~4]{evans}(Section 5.8) we have that $\varphi|_e$ is Lipschitz for every $e$. Because $\varphi$ is continuous on $\Gamma$ we have by an argument similar to~\eqref{eq:lip_proof} that $\varphi$ is Lipschitz on $\Gamma$. On the other hand let $\varphi$ be Lipschitz on $\Gamma$, hence $\varphi$ is continuous on $\Gamma$, hence bounded, and Lipschitz on every edge $e$, by \citeA[Theorem~4]{evans}(Section 5.8) we have that the weak derivative of $D(\varphi|_e)$ exists and is bounded for every edge $e$. Therefore by the definition of weak derivative on $\Gamma$ we have $\varphi\in W^{1,\infty}(\Gamma)$.
\end{proof}

\begin{remark}
    Let us conclude this section by remarking on the changes that result from adding or removing vertices of degree $2$ that are incident to two different edges on the class of functions we study. Let $z$ be a vertex in $\Gamma$ of degree $2$ that is incident to two distinct edges $e_1$ and $e_2$. Consider the metric graph $\Gamma'$ constructed from $\Gamma$ by removing the edges $e_1,e_2$ and adding the edge $e$ connecting $(\partial e_1 \cup \partial e_2)\backslash\{z\}$ with length $l_e=l_{e_1}+l_{e_2}$. We can easily construct a bijective isometry $f:\Gamma'\to\Gamma$ that satisfies $d(f(x),f(y)) = d(x,y)$ for every $x,y\in\Gamma'$. Every function $\varphi\in\mathbb{L}_p(\Gamma)$ has a corresponding function on $\Gamma'$ defined by $\varphi\mapsto \varphi\circ f$, and visa-versa, hence 
    $\mathbb{L}_p(\Gamma)\cong \mathbb{L}_p(\Gamma')$.
    However, 
    $$C^{(k),\lambda}(\Gamma)\not\cong C^{(k),\lambda}(\Gamma'),\;\text{and } W^{\alpha,p}(\Gamma)\not\cong W^{\alpha,p}(\Gamma').$$
    This is by design since we would like to allow the definition to be useful for general boundary conditions when applied to PDEs. For example, consider the graphs $\Gamma = [-1,1]$ with two vertices at $\pm1$, and $\Gamma'=[-1,0]\cup[0,1]$ with vertices at $0$ and $\pm1$. Let $\kappa>0$, and define $u:\Gamma\to\mathbb{R}$ as
    \begin{equation*}
            u(x) = e^{\kappa |t|}.
    \end{equation*}
    The function $u$ is a solution of the PDE with the generalized Kirchhoff vertex conditions \cite{mathcomp_paper}
    \begin{equation*}
        \begin{cases}
            (\kappa^2-\Delta)u = 0,\\
            \partial_+u(0) + \partial_-u(0) = 2\kappa,\\
            u(0^-) = u(0^+).
        \end{cases}
    \end{equation*}
    The function $u\in \left(C^{(k),\lambda}(\Gamma')\backslash  C^{(k),\lambda}(\Gamma)\right) \cap \left(W^{k,p}(\Gamma')\backslash  W^{k,p}(\Gamma)\right)$ for every $k,p> 1$ and $\lambda\in[0,1]$.
\end{remark}

\subsection{$H$ Spaces}
We now define the $H$ Sobolev spaces on a compact metric graph $\Gamma$, first intrinsically (Definition~\ref{def:sobolev_H_def_1}) and then via edge parameterizations (Definition~\ref{def:sobolev_H_def_2}). The latter motivates the notion of adjoint Sobolev spaces. 

\begin{definition}
For $(\alpha,p)\in\mathbb{N}\times (0,\infty)$ and $e\in\mathcal{E}$, define $H^{\alpha}(e)$ as the space of functions $\varphi\in \mathbb{L}_2(e)$ such that $\varphi\circ \zeta \in H^{\alpha}(0,l_e)$ for any natural parameterization of $e$, where $H^{\alpha}(0,l_e)$ is the fractional Sobolev space on the interval $(0,l_e)$.
\end{definition}
\citeA[Theorem~B.8]{mclean} directly implies the interpolation identification
$$H^{\alpha}(e) \cong \left[H^{\floor{\alpha}}(e), H^{\ceil{\alpha}}(e)\right]_{\alpha-\floor{\alpha}}$$
for every $\alpha\in(0,\infty)\backslash\mathbb{N}_0$.

\begin{definition}\label{def:sobolev_H_def_1}
    Let $k\in\mathbb{N}$. We define the Sobolev space $H^{k}(\Gamma)$ as the subspace of $\bigoplus_{e\in\mathcal{E}}H^{k}(e)$ consisting of functions $\varphi$ such that for every $\epsilon>0$ and isometry $\gamma:(-\epsilon,\epsilon)\to\Gamma$, the composition $(\varphi\circ\gamma)|_{(-\epsilon,\epsilon)\backslash\gamma^{-1}(\mathcal{V})}$ belongs to $H^k((-\epsilon,\epsilon)\backslash\gamma^{-1}(\mathcal{V}))$, and the even order lateral derivatives up to $k-1$ agree on the vertices, that is
    $$\partial^{2m}_+(\varphi\circ\gamma)(t) = \partial^{2m}_-(\varphi\circ\gamma)(t)$$
    for every $m\in\{0,\dots, \floor{\frac{k-1}{2}}\}$ and $t\in\gamma^{-1}(\mathcal{V})$.
\end{definition}

\begin{definition}\label{def:sobolev_H_def_2}
    Let $k\in\mathbb{N}$, and let $\eta = \{\eta\}_{e\in\mathbb{E}}$ be a parameterization of the edges of $\Gamma$. 
    \begin{enumerate}[(i)]
        \item The Sobolev space $H^k(\Gamma)$ is the subspace of $\bigoplus_{e\in\mathcal{E}}H^k(e)$ consisting of functions $\varphi$ such that
    $$D_{\eta}^{2m}\varphi_{e_1}(v) = D_{\eta}^{2m}\varphi_{e_2}(v)$$
    for every $e_1,e_2\in\mathbb{E}$, $v\in e_1\cap e_2$ and $m\in\{0,\dots,\floor{\frac{k-1}{2}}\}$.
    \item The adjoint Sobolev space $H^k_{*,\eta}(\Gamma)$ with respect the parameterization $\eta$ is defined as the subspace $\bigoplus_{e\in\mathcal{E}}H^{k}(e)$ consisting of functions $\varphi$ such that
    $$D_{\eta}^{2m-1}\varphi_{e_1}(v) = D_{\eta}^{2m-1}\varphi_{e_2}(v)$$
    for every $e_1,e_2\in\mathbb{E}$, $v\in e_1\cap e_2$, and $m\in\{1,\dots,\floor{\frac{k}{2}}\}$.
    \end{enumerate}
\end{definition}

It is straightforward to verify that Definitions~\ref{def:sobolev_H_def_1} and~\hyperref[def:sobolev_H_def_2]{\getrefnumber{def:sobolev_H_def_2}(i)} are equivalent for $H^k(\Gamma)$. Moreover, the following duality-type relations hold for any parameterization $\eta$:
$$\varphi\in H^k(\Gamma) \Rightarrow D_{\eta}\varphi\in H^{k-1}_{*,\eta}(\Gamma),\quad \varphi\in H^k_{*,\eta}(\Gamma) \Rightarrow D_{\eta}\varphi\in H^{k-1}(\Gamma).$$
Observe that $\varphi\in H^k(\Gamma)$ implies $D_{\eta}\varphi\in H^{k-1}_{*,\eta}(\Gamma)$ for any parameterization $\eta$. Furthermore, we also have the immediate identifications
\begin{align}
    H^1_{*,\eta}(\Gamma)&\cong \bigoplus_{e\in\mathcal{E}} H^{1}(e),\label{eq:H_dual_ident_1}\\
    H^1(\Gamma)&\cong \bigoplus_{e\in\mathcal{E}} H^{1}(e)\cap C(\Gamma).\label{eq:H_1_ident_1}
\end{align}
We set $H^0(\Gamma)=H^0_{*,\eta}(\Gamma)=\mathbb{L}_2(\Gamma)$.
\begin{definition}
    Let $\alpha\in(0,1)$ and let $\eta$ be a parameterization of the edges of $\Gamma$, define the fractional Sobolev space $H^{\alpha}(\Gamma)$ and the fractional adjoint Sobolev space $H^{\alpha}_{*,\eta}(\Gamma)$ via the real interpolation 
    $$H^{\alpha}(\Gamma)=[\mathbb{L}_2(\Gamma), H^1(\Gamma)]_{\alpha},$$
    $$H^{\alpha}_{*,\eta}(\Gamma)=[\mathbb{L}_2(\Gamma), H^1_{*,\eta}(\Gamma)]_{\alpha}.$$
    For $\alpha>1$, if $\floor{\alpha}$ is odd, define $H^{\alpha}(\Gamma)$ as the space of functions $\varphi\in H^{\floor{\alpha}}(\Gamma)$ such that $D^{\floor{\alpha}}_{\eta}\varphi\in H^{\alpha-\floor{\alpha}}_{*,\eta}(\Gamma)$ for every parameterization $\eta$ of the graph $\Gamma$. We equip $H^{\alpha}(\Gamma)$ with the norm
    $$\|\varphi\|^2_{H^{\alpha}(\Gamma)} = \|\varphi\|^2_{H^{\floor{\alpha}}(\Gamma)} + \|D^{\floor{\alpha}}\varphi\|^2_{H^{\alpha-\floor{\alpha}}_{*,\eta}(\Gamma)}.$$
    If $\floor{\alpha}$ is even, $H^{\alpha}(\Gamma)$ is the space of functions $\varphi$ such that $D^{\floor{\alpha}}_{\eta}\varphi=D^{\floor{\alpha}}\varphi\in H^{\alpha-\floor{\alpha}}(\Gamma)$, and we equip $H^{\alpha}(\Gamma)$ with the norm
    $$\|\varphi\|^2_{H^{\alpha}(\Gamma)} = \|\varphi\|^2_{H^{\floor{\alpha}}(\Gamma)} + \|D^{\floor{\alpha}}\varphi\|^2_{H^{\alpha-\floor{\alpha}}(\Gamma)}.$$
\end{definition}
The goal of the next few results is to establish the equivalence between the interpolation theoretic definition of $H^{\alpha}(\Gamma)$  with the norm based $W^{\alpha,2}(\Gamma)$, validating both as natural frameworks for analysis on metric graphs. We will require the following standard result from interpolation theory in the following analysis. 
\begin{lemma}\label{lem:direct_sum_interpolation_lemma}
    Let $n\in\mathbb{N}$, and let $H$ be a vector space. For each $k\in\{1,\dots,n\}$ let $H_{0,k}$ and $H_{1,k}$ be Hilbert spaces with $H_{1,k}\hookrightarrow H_{0,k}\subset H$. Then, for every $\alpha\in(0,1)$, the interpolation space satisfies
    $$\left[\bigoplus_{k=1}^n H_{0,k}, \bigoplus_{k=1}^n H_{1,k}\right]_{\alpha} \cong \bigoplus_{k=1}^n \left[H_{0,k},H_{1,k}\right]_{\alpha}.$$
\end{lemma}
\begin{proof}
    We consider the K-method of interpolation using the norms $$\|\varphi\|_{\bigoplus_{k=1}^n H_{j,k}}^2 = \sum_{k=1}^n\|\varphi_k\|_{H_{j,k}}^2$$ 
    for $j\in\{1,2\},$ where $\varphi=\sum_{k=1}^n\varphi_k$ uniquely with $\varphi_k\in H_{j,k}$. Let $u\in\bigoplus_{k=1}^n H_{0,k}$. The $K$ function is
    \begin{equation*}
        \begin{split}
            & \mathrel{\hphantom{=}} K(t,u)\\
            &=\inf\Bigg\{\|\phi\|_{\bigoplus_{k=1}^n H_{0,k}}^2 + t^2\|\psi\|_{\bigoplus_{k=1}^n H_{1,k}}^2: \phi\in \bigoplus_{k=1}^n H_{0,k},\psi\in \bigoplus_{k=1}^n H_{1,k},\\
            &\hphantom{=\inf\Bigg\{\|\phi\|_{\bigoplus_{k=1}^n H_{0,k}}^2 + t^2\|\psi\|_{\bigoplus_{k=1}^n H_{1,k}}^2:}\varphi+\psi=u\Bigg\}\\
            &= \inf\Bigg\{\sum_{k=1}^n\left(\|\varphi_k\|^2_{H_{0,k}} + t^2 \|\psi_k\|^2_{H_{1,k}}\right)  : \varphi_k\in H_{0,k}, \psi\in H_{1,k},\varphi_k + \psi_k = u_k\\
            &\hphantom{=\inf\Bigg\{\sum_{k=1}^n\left(\|\varphi_k\|^2_{H_{0,k}} + t^2 \|\psi_k\|^2_{H_{1,k}}\right)  :}\forall k\in\{1,\dots,n\}\Bigg\}\\
            &= \sum_{k=1}^n\inf\left\{\|\varphi_k\|^2_{H_{0,k}} + t^2 \|\psi_k\|^2_{H_{1,k}}  : \varphi_k\in H_{0,k}, \psi\in H_{1,k},\varphi_k + \psi_k = u_k\right\}\\
        \end{split}
    \end{equation*}
    The last equality is the sum of the $K$ functions associated with the interpolation $[H_{0,k},H_{1,k}]$. It is therefore clear that $u=\sum_{k=1}^nu_k$ with $u_k\in H_{0,k}$ belonging to $$\left[\bigoplus_{k=1}^n H_{0,k}, \bigoplus_{k=1}^n H_{1,k}\right]_{\alpha}$$ if and only if $u_k$ belong to $\left[H_{0,k},H_{1,k}\right]_{\alpha}$ for every $k\in\{1,\dots,n\}$.
\end{proof}

\begin{prop}\label{prop:H_adjoint_characterizaton_1}
Let $\alpha\in[0,1]$ and let $\eta$ be a parameterization of $\Gamma$. Then
    \begin{equation}\label{eq:adjoint_H_alpha_1}
        H^{\alpha}_{*,\eta}(\Gamma)\cong \bigoplus_{e\in\mathcal{E}}H^{\alpha}(e)
    \end{equation}
\end{prop}
\begin{proof}
    Recalling the identification \eqref{eq:H_dual_ident_1}, \eqref{eq:adjoint_H_alpha_1} follows directly from Lemma~\ref{lem:direct_sum_interpolation_lemma} applied to $$\mathbb{L}_2(\Gamma)\cong\bigoplus_{e\in\mathcal{E}}\mathbb{L}_2(e),\quad H^{1}_{*,\eta}(\Gamma)\cong\bigoplus_{e\in\mathcal{E}}H^{1}(e).$$
\end{proof}

\begin{theorem}\label{thm:H_W_equivalence}
    Let $p>0$ and $\alpha\in(0,\infty)\backslash\{2k+\frac{1}{2} : k\in\mathbb{N}\}$. We have 
    $H^{\alpha}(\Gamma)\cong W^{\alpha,2}(\Gamma)$.
\end{theorem}
\begin{proof}
    By standard results (\citeA[Theorem~3.18]{mclean} and \citeA[Lemma~4.2, Theorem~4.6]{chandler}) we have $H^{\alpha}(e)\cong W^{\alpha,2}(e)$ for each edge $e$ and $\alpha>0$. By \citeA[Theorem~4.1]{mathcomp_paper} we have the characterizations 
    $$H^{\alpha}(\Gamma)\cong\bigoplus_{e\in\mathcal{E}} H^{\alpha}(e),\quad \text{for}\;\alpha\in(0,1/2),$$
    $$H^{\alpha}(\Gamma)\cong\bigoplus_{e\in\mathcal{E}} H^{\alpha}(e) \cap C(\Gamma),\quad \text{for}\;\alpha\in(1/2,1].$$ 
    Therefore, by combining the above facts with Theorem~\ref{thm:W_chatacterization_3} we have $$H^{\alpha}(\Gamma)\cong W^{\alpha,2}(\Gamma).\quad \text{for}\;\alpha\in(0,\frac{1}{2})\cup(\frac{1}{2},1].$$
    If $\alpha\in\mathbb{N}$, then by the definitions of $H^{\alpha}(\Gamma)$ and $W^{\alpha,2}(\Gamma)$, the fact that $H^{\alpha}(e)\cong W^{\alpha,2}(e)$ for every edge $e$, and Theorem~\ref{thm:W_imbedding_1} (applied to the metric graph $e$), we have 
    \begin{equation*}
    \begin{split}
        H^{\alpha}(\Gamma)&\cong \bigoplus_{e\in\mathcal{E}}H^{\alpha}(e)\cap C^{\left(2\floor{\frac{\alpha-1}{2}}\right)}(\Gamma)\cong \bigoplus_{e\in\mathcal{E}}W^{\alpha,2}(e)\cap C^{\left(2\floor{\frac{\alpha-1}{2}}\right)}(\Gamma)\\&\cong W^{\alpha,2}(\Gamma).
    \end{split}
    \end{equation*}
    Let $\alpha\in(0,\infty)\backslash\left(\mathbb{N}\cup\{2k+\frac{1}{2}:k\in\mathbb{N}_0\}\right)$. If $\floor{\alpha}$ is odd, then $\varphi\in H^{\alpha}(\Gamma)$ is equivalent to
    $\varphi\in H^{\floor{\alpha}}(\Gamma)\cong W^{\floor{\alpha},2}(\Gamma)$ and $D^{\floor{\alpha}}\varphi\in H^{\alpha-\floor{\alpha}}_{*,\eta}(\Gamma)\cong \bigoplus_{e\in \mathcal{E}}H^{\alpha-\floor{\alpha}}(e) \cong \bigoplus_{e\in \mathcal{E}}W^{\alpha-\floor{\alpha},2}(e)$ for any parameterization $\eta$ by using Proposition~\ref{prop:H_adjoint_characterizaton_1}. These conditions coincide with the definition of $W^{\alpha,p}(\Gamma)$ when $\floor{\alpha}$ is odd, thus $H^{\alpha}(\Gamma)= W^{\alpha,2}(\Gamma)$, and the above identifications further yield equivalence of the respective norms. Consequently,
    $H^{\alpha}(\Gamma)\cong W^{\alpha,2}(\Gamma)$.
    Now suppose $\floor{\alpha}$ is even. Then $\varphi\in H^{\alpha}(\Gamma)$ if and only if $\varphi\in H^{\floor{\alpha}}(\Gamma)\cong W^{\floor{\alpha}}(\Gamma)$ and $D^{\floor{\alpha}}\varphi\in H^{\alpha-\floor{\alpha}}(\Gamma)\cong W^{\alpha-\floor{\alpha}}(\Gamma)$, these conditions on $\varphi$ are equivalent to the definition of $W^{\alpha,2}(\Gamma)$, and similarly to the previous case we have
    $H^{\alpha}(\Gamma)\cong W^{\alpha,2}(\Gamma)$.
\end{proof}

\section{Bounds for the Eigenfunctions of the Laplacian}
\label{sec:eigen_bounds}
In this section, we present uniform bounds for the eigenfunctions of a class of Laplacians on metric graphs; these bounds will help apply the Sobolev spaces developed in the previous sections. The result of this section is of individual importance, although we will use it as an auxiliary result in this article. Consider the Laplacian $\Delta_{\Gamma}$ on the metric graph $\Gamma$ defined by $\Delta_{\Gamma} = \bigoplus_{e\in\mathcal{E}}\Delta_e$ where $\Delta_eu=(\partial^2 (u\circ\zeta))\circ\zeta^{-1}$ for every $u\in H^2(e)$, where $\zeta$ is a natural parameterization of $e$. This means that for every $x\in e$ we define $\Delta_{\Gamma}u(x)=\Delta_eu(x)$. We couple $\Delta_{\Gamma}$ with vertex boundary conditions that makes $(\Delta_{\Gamma},\mathcal{D}(\Delta_{\Gamma})):\mathcal{D}(\Delta_{\Gamma})\subset\mathbb{L}_2(\Gamma)\to\mathbb{L}_2(\Gamma)$ a self-adjoint operator on $\mathbb{L}_2(\Gamma)$, where $\mathcal{D}(\Delta_{\Gamma})$ is the domain of definition of $\Delta_{\Gamma}$ under these boundary conditions. We further require
\begin{itemize}
    \item $(\Delta_{\Gamma},\mathcal{D}(\Delta_{\Gamma}))$ has a compact resolvent.
    \item There exists $\rho > 0$ such that the operator $\rho-\Delta_{\Gamma}$ is positive.
\end{itemize}
These assumptions are satisfied, for instance, by a mixture of the following standard vertex conditions imposed at each vertex $v$
\begin{itemize}
    \item Kirchhoff condition $\sum_{(e,\zeta)\in\textrm{Dir}(v)}\partial_{e,\zeta}u(v)=0$.
    \item Dirichlet condition $u(v)=0$.
    \item Generalized Kirchhoff condition $\sum_{(e,\zeta)\in\textrm{Dir}(v)}\partial_{e,\zeta}u(v)=\alpha u(v)$,\\ where $\alpha\in\mathbb{R}$, under appropriate assumptions \cite{mathcomp_paper}.
\end{itemize}

Under these requirements \citeA[Theorem~1.7.16 and Theorem~2.6.8]{miklavcic} guarantee the existence of two sequences: a sequence of eigenfunctions $\{\varphi_i\}_{i\in\mathbb{N}}$ in $\mathcal{D}(\Delta_{\Gamma})$ that constitute an orthonormal basis for $\mathbb{L}_2(\Gamma)$, and a non-negative sequence of corresponding eigenvalues $\{\lambda_i\}_{i\in\mathbb{N}}$ that satisfy $\Delta_{\Gamma}\varphi_i=-\lambda_i\varphi$ and $\lim_{i\to\infty}\lambda_i=\infty$.

\begin{lemma}\label{lem:smooth_eigenfunctions}
    Let $\delta>0$, $\gamma:(-\delta,\delta)\to\Gamma$ be a local isometry, and $\varphi$ be an eigenfunction  $-\Delta_{\Gamma}$ with eigenvalue $\lambda$. The function $\tilde{\varphi}=\varphi\circ\gamma$ satisfies $\Delta\tilde{\varphi} = -\lambda \tilde{\varphi}$ on $(-\delta,\delta)\backslash\gamma^{-1}(\mathcal{V})$ and $\tilde{\varphi}|_{(-\delta,\delta)\backslash \gamma^{-1}(\mathcal{V})}\in C^{\infty}((-\delta,\delta)\backslash \gamma^{-1}(\mathcal{V}))$. Moreover, if $t\in(-\delta,\delta)$ such that $\gamma(t)\in\mathcal{V}$ and satisfies $\partial_{e_1}\varphi(\gamma(t)) = -\partial_{e_2}\varphi(\gamma(t))$ where $e_1,e_2$ are defined by
    $$\{e_1,e_2\} = \bigcap_{s>0}\{e\in\mathcal{E} : e\cap\gamma((t-s,t+s))\not=\varnothing\},$$
    then $\Delta\tilde{\varphi} = -\lambda \tilde{\varphi}$ holds at $t$ and the function $\tilde{\varphi}$ is $C^{\infty}$ in a neighborhood of $t$.
\end{lemma}
\begin{proof}
    Without loss of generality, assume $\gamma$ is an isometry.
    Noting that the set $\gamma^{-1}(\mathcal{V})$ is finite, we just prove that $\Delta \tilde{\varphi} =-\lambda\tilde{\varphi}$ on $(-\delta,\delta)\backslash \gamma^{-1}(\mathcal{V})$ as the smoothness result will follow from \citeA[Theorem~3]{evans} (Section 6.3). Let $t\in(-\delta,\delta)\backslash\gamma^{-1}(\mathcal{V})$, and let $\epsilon>0$ such that $B_{2\epsilon}(\gamma(t))\cap\mathcal{V}=\varnothing$, where $B_{r}(x)=\{y\in\Gamma: d(x,y)<r\}$. Let $\zeta:B_{2\epsilon}(\gamma(t))\to (-\epsilon,\epsilon) $ be an isometry that satisfies $\zeta(x)=0$. By the definition of $\Delta_{\Gamma}$ it is immediate that
    $$(\Delta_{\Gamma}\varphi)(\zeta^{-1}(y)) = \Delta (\varphi\circ\zeta^{-1})(y)$$
    for all $y\in(-\epsilon,\epsilon)$, hence $\Delta (\varphi\circ\zeta^{-1}) = -\lambda (\varphi\circ\zeta^{-1})$. Because both $s\mapsto\gamma(t+s)$ and $\zeta^{-1}$ are isometries on $(-\epsilon,\epsilon)$ either $\gamma|_{(-\epsilon,\epsilon)}=\zeta^{-1}$ or $\gamma|_{(-\epsilon,\epsilon)}(t+s)=\zeta^{-1}(-s)$ for all $s\in(-\epsilon,\epsilon)$. Substituting, we have $$\Delta(\varphi\circ\gamma)(t) = -\lambda(\varphi\circ\gamma)(t)$$
    on $(-\epsilon,\epsilon)$. To show the second part let $t\in(-\delta,\delta)$ such that $\gamma(t)\in\mathcal{V}$, notice that the degree of $\gamma(t)$ is at least $2$. Let $\epsilon>0$ such that $B_{2\epsilon}(\gamma(t))\cap\mathcal{V}=\{\gamma(t)\}$. and let $\zeta_1:e_1\to[0,l_1]$ and $\zeta_2:e_2\to[0,l_2]$ be different parameterizations that satisfy $\zeta_1(\gamma(0))=\zeta_2(\gamma(0))=0$. Define the map
    $$\zeta(y)=\begin{cases}
        -\zeta_1(y)\quad, y\in e_1,\\
        \zeta_2(y)\quad, y\in e_2.
    \end{cases}$$
    Similar to the argument above, it is enough to show that the result for $\zeta$ instead of $\gamma$. Suppose we have $\partial_{e_1}\varphi(\zeta^{-1}(0)) = -\partial_{e_2}\varphi(\zeta^{-1}(0))$, by the definition of inward derivatives, this condition is equivalent to 
    \begin{equation}\label{eq:lateral_equal_1}
        \partial_+ (\varphi\circ\zeta^{-1})(0) = \partial_- (\varphi\circ\zeta^{-1})(0).
    \end{equation}
    Thus $\varphi\circ\zeta^{-1}$ is differentiable. By \citeA[Theorem~8.14]{gilbarg_trudinger} we have that 
    $\varphi\circ\zeta^{-1}\in C^{\infty}([0,l_{e_2}])\cap C^{\infty}([-l_{e_1},0])$.
    Therefore, the value of $(\Delta_{\Gamma}\varphi)(\zeta^{-1}(0))$ is well-defined, which is understood (by definition) as lateral derivatives, more precisely
    \begin{equation*}
        \begin{split}
            &\mathrel{\hphantom{=}}\lim_{s\to0^+}\frac{\partial(\varphi\circ\zeta^{-1})(s) - \partial_+(\varphi\circ\zeta^{-1})(0)}{s} \\
            &= \lim_{s\to0^-}\frac{\partial(\varphi\circ\zeta^{-1})(s) - \partial_-(\varphi\circ\zeta^{-1})(0)}{s}.
        \end{split}
    \end{equation*}
    Because of the equality \eqref{eq:lateral_equal_1} we have that $\varphi\circ\zeta^{-1}$ is twice differentiable and its derivative is equal to $\Delta_{\Gamma}\varphi(\zeta^{-1}(0))$. Therefore, we have the equality
    $$\Delta (\varphi\circ\zeta^{-1})(0) = (\Delta_{\Gamma}\varphi)(\zeta^{-1}(0)) = -\lambda (\varphi)(\zeta^{-1}(0)) = -\lambda (\varphi\circ\zeta^{-1})(0).$$
\end{proof}

\begin{remark}
    The loop edge case $e_1=e_2$ is included in the proof of Lemma~\ref{lem:smooth_eigenfunctions}.
\end{remark}
The following result is the main result of the section.
\begin{theorem}\label{thm:bounded_eigenfunctions}
    The eigenfunctions $\{\varphi_i\}_{i\in\mathbb{N}}$ of $\Delta_{\Gamma}$ are uniformly bounded:
    $$\sup_{i\in\mathbb{N}}\sup_{x\in\Gamma}|\varphi_i(x)| < \infty.$$
\end{theorem}
\begin{proof}
     Let $\mathcal{I}=\{i\in\mathbb{N}:\lambda_i>0\}$. It is enough to show the results for $i\in\mathcal{I}$ since by the compactness of $\Gamma$ and the finiteness of the set $\mathcal{I}$ we have
     $$\sup_{i\in\mathbb{N}\backslash\mathcal{I}}\sup_{x\in\Gamma}|\varphi_i(x)|<\infty.$$ 
     Let $0<m<\min_{e\in\mathcal{E}}l_e$, we will show that for every $x\in\Gamma$ there exists a neighborhood $\mathcal{N}$ of $x$ for which $$\sup_{y\in\mathcal{N}}\sup_{i\in\mathcal{I}}|\varphi_i(y)|<\infty.$$
     The desired result follows from the compactness of $\Gamma$ by choosing a finite subcover from these neighborhoods.\\
     Define $\mathcal{V}_1$ to be the set of vertices of degree $1$. First we consider the case $x\in\Gamma\backslash\mathcal{V}_1$ and $i\in\mathcal{I}$, let $0<\epsilon<m$ such that $B_{4\epsilon}(x)\cap\mathcal{V}\subset\{x\}$. Let $\zeta:B_{4\epsilon}(x)\to (-2\epsilon,2\epsilon) $ be an isometry that satisfies $\zeta(x)=0$. Define the one-sided natural maps $\zeta_1 = \zeta|_{\zeta^{-1}\left([0,2\epsilon)\right)}$ and $\zeta_2 = \zeta|_{\zeta^{-1}\left((2\epsilon,0]\right)}$. Denote $\tilde{\varphi}_i = \varphi_i\circ\zeta_1^{-1}$ (The argument below applies to $\zeta_2$ as well). By Lemma~\ref{lem:smooth_eigenfunctions} we have $\tilde{\varphi_i}\in C^{\infty}((0,2\epsilon))$ and $\Delta\tilde{\varphi_i}=-\lambda_i\tilde{\varphi}_i$ on $(0,2\epsilon)$. Let $\eta:\mathbb{R}\to\mathbb{R}_+$ be the non-negative $C^{\infty}_c(\mathbb{R})$ function defined by
    $$\eta(x) = e^{\frac{\epsilon^2}{x^2-\epsilon^2}+1}\mathbbm{1}_{(-\epsilon,\epsilon)(x)}.$$
    Define the operator $Q=\lambda_i+1 + \Delta$, the function $F$ defined by $F(x)=\frac{1}{2\sqrt{\lambda_i+1}}\sin(\sqrt{\lambda_i+1}|x|)$ is the fundamental solution of $Q$. Let $y\in(0,\epsilon/2)$, integrating by parts on the interval $(0,\epsilon)$ we have
    \begin{equation*}
        \begin{split}       &\mathrel{\hphantom{=}}\int_0^{\epsilon}\eta(z) F(z)\tilde{\varphi}_i(z+y)dz\\
        &=\int_0^{\epsilon}\eta(z) F(z) Q\tilde{\varphi}_i(z+y)dz\\
        &= \eta(0)F(0)D\tilde{\varphi}(y) + \left(D\eta(0)F(0) + \eta(0)\partial_+F(0)\right)\tilde{\varphi}(y)\\&+\int_{0}^{\epsilon} Q(\eta F)(z)\tilde{\varphi}_i(z+y)dz\\
        &=\frac{1}{2}\tilde{\varphi}(y)+\int_0^{\epsilon}\Delta\eta(z)F(z)\tilde{\varphi}_i(z+y)dz + 2\int_0^{\epsilon}D\eta(z)D F(z)\tilde{\varphi}_i(z+y)dz.
        \end{split}
    \end{equation*}
    Rearranging,
    \begin{equation*}
        \begin{split}
            \tilde{\varphi}_i(y) &=- 2\int_0^{\epsilon} \Delta\eta(z)F(z)\tilde{\varphi}_i(z+y)dz - 4\int_0^{\epsilon} D\eta(z)D F(z)\tilde{\varphi}_i(z+y)dz \\&+ 2\int_0^{\epsilon}\eta(z) F(z)\tilde{\varphi}_i(z+y)dz.
        \end{split}
    \end{equation*}
    From the definition of $\eta$ we see that $\eta(0)=1$, $\eta\leq 1$, $|D\eta|\leq C\frac{1}{\epsilon}$, and $|\Delta\eta|\leq C\frac{1}{\epsilon^2}$ where $C$ is independent of $\epsilon$.
    Using these aforementioned facts and the Cauchy-Schwarz inequality, we have
    \begin{equation}\label{eq:phi_bound_1}
        \begin{split}
            |\tilde{\varphi}_i(y)|&\leq 2C\frac{1}{\epsilon^2}\|\mathbbm{1}_{(-\epsilon,\epsilon)}F\|_{\mathbb{L}_2}\|\tilde{\varphi}_i\|_{\mathbb{L}_2} 
            +4C\frac{1}{\epsilon}\|\mathbbm{1}_{(-\epsilon,\epsilon)}D F\|_{\mathbb{L}_2}\|\tilde{\varphi}_i\|_{\mathbb{L}_2} \\&+ 2\|\mathbbm{1}_{(-\epsilon,\epsilon)}F\|_{\mathbb{L}_2}\|\tilde{\varphi}_i\|_{\mathbb{L}_2}.
        \end{split}
    \end{equation}
    By direct calculations, we have
    \begin{equation*}
        \begin{split}
            \|\mathbbm{1}_{(-\epsilon,\epsilon)}F\|_{\mathbb{L}_2}^2=\frac{1}{4(\lambda_i+1)}\int_{-\epsilon}^{\epsilon} \sin(\sqrt{\lambda_i+1}|x|)^2 dx\leq\frac{\epsilon}{2(
            \lambda_i+1)},
        \end{split}
    \end{equation*}
    and 
    \begin{equation*} 
        \begin{split}
            \|\mathbbm{1}_{(-\epsilon,\epsilon)}D F\|_{\mathbb{L}_2}^2&=\frac{1}{4}\int_{-\epsilon}^{\epsilon} \cos(\sqrt{\lambda_i+1}|x|)^2 dx\leq\frac{\epsilon}{2}.
        \end{split}
    \end{equation*}
    Finally because $\|\tilde{\varphi}_i\|_{\mathbb{L}_2}\leq\|\varphi_i\|_{\mathbb{L}_2}=1$ we have a bound of the form
    $|\tilde{\varphi}_i(y)|\leq C(\epsilon)$
    for all $y\in(0,\epsilon/2)$, furthermore, the bound applies to $y=0$ by continuity of $\tilde{\varphi}_i$. Applying the same argument to $\varphi_i\circ\zeta_2^{-1}$, we have 
    $$\sup_{y\in B_{\epsilon/2}(x)}|\varphi_i(x)|<C(\epsilon).$$
    Since the argument above is in essence one-sided, the same argument applies for $x\in\mathcal{V}_1$. Compactness finishes the proof, as mentioned at the beginning of the proof.
\end{proof}

\begin{remark}
    In the proof of Theorem~\ref{thm:bounded_eigenfunctions} we did not use the ``moreover" part of Lemma~\ref{lem:smooth_eigenfunctions}.
\end{remark}
\begin{corr}\label{corr:eigenfunctions_derivative_bound}
    There exists $C<\infty$ such that for every $(i,j)\in\mathbb{N}\times\mathbb{N}$ and parameterization $\eta$ of $\Gamma$ the following inequality holds
    $$\sup_{x\in\Gamma}|D_{\eta}^j\varphi_i(x)|\leq C\lambda_i^{j/2}.$$
\end{corr}
\begin{proof}
    By Lemma~\ref{lem:smooth_eigenfunctions}, it is enough to show the inequality for the interior of edges. Let $e\in\mathcal{E}$, $\zeta$ a parameterization of $e$, and $i,j\in\mathbb{N}$. Let $a=\varphi_i(\zeta(0))$ and $b=\varphi_i(\zeta(l_e))$. By Lemma~\ref{lem:smooth_eigenfunctions} the function $\psi=\varphi_i\circ\zeta$ satisfies $\Delta\psi = -\lambda_i\psi$ with the boundary conditions $\psi(0)=a$ and $\psi(l_e)=b$. Therefore 
    \begin{equation*}
        \psi(t) = \begin{cases}
            a\cos(\sqrt{\lambda_i}t)+\frac{b-a\cos(\sqrt{\lambda_i}l_e)}{\sin(\sqrt{\lambda_i}l_e)}\sin(\sqrt{\lambda_i}t), \quad b\not=a\cos(\sqrt{\lambda_i}l_e),\\
            a\cos(\sqrt{\lambda_i}t),\quad \textrm{otherwise}.
        \end{cases}
    \end{equation*}
    for all $t\in(0,l_e)$. Calculating $D^j\psi$ and using Theorem~\ref{thm:bounded_eigenfunctions} yields the inequality.
\end{proof}

\section{Applications to the Fractional Operator $(\kappa^2-\Delta_{\Gamma})^{\alpha/2}$}
\label{sec:appli_1}
In this section, we characterize the spaces $\dot{H}^{\alpha}$ defined below as the domains of the fractional operators $(\kappa^2-\Delta_{\Gamma})^{\alpha/2}$. We show that these characterization yield regularity results for both the fractional PDE $(\kappa^2-\Delta_{\Gamma})^{\alpha/2}u=f$ and the fractional SPDE $(\kappa^2-\Delta_{\Gamma})^{\alpha/2}u=\mathcal{W}$ with $\mathcal{W}$ being a Gaussian white noise on $\mathbb{L}_2(\Gamma)$. The latter SPDE is used to construct Whittle-Mat\'ern fields on metric graphs \cite{bern_gaussian_matern}. For $\alpha=1$, this coincides with the massive Gaussian free fields (massive GFFs), defined as a centered Gaussian on $\Gamma$ with covariance operator $(\kappa^2-\Delta_{\Gamma})^{-1}$ \cite{sheffield_2007, rodriguez_2016}. In this case, the space $\dot{H}^1$ is the Cameron-Martin space associated with the massive GFF.\\
Let $\mathcal{V}_D$ and $\mathcal{V}_K$ form a partition of vertex set $\mathcal{V}$. Define the set of edges incident to $\mathcal{V}_K$ as $\mathcal{E}_K=\{e\in\mathcal{E}:e\cap\mathcal{V}_K\not=\varnothing\}$, and the set of edges incident to $\mathcal{V}_D$ as $\mathcal{E}_D=\{e\in\mathcal{E}:e\cap\mathcal{V}_D\not=\varnothing\}$. We introduce the following functional spaces that will represent the Kirchhoff\footnote{Kirchhoff conditions may also be referred to as Neumann conditions.} and Dirichlet boundary conditions at the vertices:
\begin{equation*}
    \begin{split}
        &K(\Gamma)=\\
        &\left\{ u\in\bigcup_{s>\frac{3}{2}}\bigoplus_{e\in\mathcal{E}_K}H^{s}(e)\oplus \bigoplus_{e\in\mathcal{E}\backslash\mathcal{E}_K}\mathbb{L}_2(e) : \sum_{(e,\zeta)\in\textrm{Dir}(v)}\partial_{e,\zeta}u(v)=0 \quad\forall v\in\mathcal{V}_K\right\},
    \end{split}
\end{equation*}
$$D(\Gamma)=\left\{ u\in\bigcup_{s>\frac{1}{2}}\bigoplus_{e\in\mathcal{E}_D}H^{s}(e)\oplus \bigoplus_{e\in\mathcal{E}\backslash\mathcal{E}_D}\mathbb{L}_2(e) : u(v)=0 \quad\forall v\in\mathcal{V}_D\right\}.$$
Let $\Delta_{\Gamma}$ be the Laplacian operator $\bigoplus_{e\in\mathcal{E}}\Delta_e$ with domain $$\mathcal{D}(\Delta_{\Gamma})=\bigoplus_{e\in\mathcal{E}}H^{2}(e)\cap K(\Gamma)\cap D(\Gamma).$$ Let $\kappa>0$ and consider the operator $L=\kappa^2-\Delta_{\Gamma}$, which is a positive and self-adjoint. Furthermore, $L$ has a compact resolvent. Indeed, the spectrum of $L$ is real and hence for any complex number $c\in\mathbb{C}\backslash\mathbb{R}$ the operator $(L-c)^{-1}$ is well-defined and continuous from $\mathbb{L}_2(\Gamma)$ into $\bigoplus_{e\in\mathcal{E}}H^{2}(e)$. Hence by the Rellich–Kondrachov theorem the operator $(L-c)^{-1}:\mathbb{L}_2(\Gamma)\to\mathbb{L}_2(\Gamma)$ is compact. Therefore, by \citeA[Theorem~1.7.16 and Theorem~2.6.8]{miklavcic}, there exist sequences of positive real eigenvalues $\{\lambda_i\}_{i\in\mathbb{N}}$ of $L$ that increase to infinity and corresponding eigenfunctions $\{\varphi_i\}_{i\in\mathbb{N}}$ that constitute a basis over $\mathbb{L}_2(\Gamma)$. Moreover, by Lemma~\ref{lem:smooth_eigenfunctions} each eigenfunction satisfies $\varphi_i\in\bigoplus_{e\in\mathcal{E}}C^{\infty}(e)$. Write $\widehat{\lambda}_i = \lambda_i-\kappa^2$, and note that $\Delta \varphi_i=-\widehat{\lambda}_i\varphi_i$.
\begin{theorem}[Weyl's law]
There exist constants $A,B>0$ that satisfy
    $Ai^2\leq\lambda_i\leq Bi^2$ for all $i\in\mathbb{N}$.
\end{theorem}
\begin{proof}
    For positive integers $k$ denote $N(k)=|\{i\in\mathbb{N}: \lambda_i\leq k^2\}|$. Since the graph $\Gamma$ is compact, by \citeA[Lemma~4.4]{berk} there exists $C_1,C_2>0$ such that 
    $C_1k\leq N(k)\leq C_2k$ for all $k\in\mathbb{N}$. The result follows by noticing that for all $j\in\mathbb{N}$ we have $N(\floor{\sqrt{\lambda_j}}-1)\leq j$ and $N(\floor{\sqrt{\lambda_j}}+1)\geq j$.
\end{proof}

Using the eigenfunctions we define for $\alpha>0$ the fractional power operator $L^{\alpha/2}$ by 
\begin{equation}\label{eq:L_alpha_def_1}
L^{\alpha/2}u = \left(\kappa^2-\Delta_{\Gamma}\right)^{\alpha/2}u = \sum_{i=1}^{\infty}\lambda_i^{\alpha/2}\langle u,\varphi_i\rangle\varphi_i, 
\end{equation}
on the space $$\mathcal{D}(L^{\alpha/2})=\dot{H}^{\alpha}=\left\{u\in\mathbb{L}_2(\Gamma) : \sum_{i=1}^{\infty}\lambda_i^{\alpha}\langle u,\varphi_i\rangle^2<\infty\right\},$$
where the sum in \eqref{eq:L_alpha_def_1} converges in $\mathbb{L}_2(\Gamma)$ whenever $u\in\dot{H}^{\alpha}$.
Equip the space $\dot{H}^{\alpha}$ by the norm $\|u\|_{\alpha}^2 = \sum_{i=1}^{\infty}\lambda_i^{\alpha}\langle u,\varphi_i\rangle^2$. The dual space of $\dot{H}^{\alpha}$ is denoted $\dot{H}^{-\alpha}$ and we equip it with the equivalent (to the operator norm) norm $\|\psi\|_{-\alpha}^2=\sum_{i=1}^{\infty}\lambda_i^{-\alpha}\langle\psi,\varphi_i\rangle^{2}$,  with $\langle\cdot,\cdot\rangle$ meaning the paring in this context. The operator $L^{\alpha/2}$ defined in \eqref{eq:L_alpha_def_1} can be isometrically (hence continuously) extended to $L^{\alpha/2}:\dot{H}^{\beta}\to\dot{H}^{\beta-\alpha}$ for any $\beta\in\mathbb{R}$, this extension is easily seen to be surjective. This justifies defining the continuous inverse operator $L^{-\alpha/2}:\dot{H}^{\beta}\to\dot{H}^{\beta+\alpha}$. We now aim to characterize the spaces $\dot{H}^{\alpha}$ in terms of the Sobolev spaces introduced earlier; such identifications have been pursued in \citeA{ mathcomp_paper, bern_gaussian_matern}. For each $n\in\mathbb{N}$, define the spaces
\begin{equation*}
    \begin{split}
        K^n(\Gamma)&=\Bigg\{ u\in\bigcup_{s>2n-\frac{1}{2}}\bigoplus_{e\in\mathcal{E}_K}H^{s}(e)\oplus \bigoplus_{e\in\mathcal{E}\backslash\mathcal{E}_K}\mathbb{L}_2(e) : D^{2j}u\in K(\Gamma) \\
        &\hphantom{=\Bigg\{ u\in\bigcup_{s>2n-\frac{1}{2}}\bigoplus_{e\in\mathcal{E}_K}H^{s}(e)\oplus \bigoplus_{e\in\mathcal{E}\backslash\mathcal{E}_K}\mathbb{L}_2(e) :} \forall j\in\{0,\dots, n-1\}\Bigg\},
    \end{split}
\end{equation*}
\begin{equation*}
    \begin{split}
        D^n(\Gamma)&=\Bigg\{ u\in\bigcup_{s>2n-\frac{3}{2}}\bigoplus_{e\in\mathcal{E}_D}H^{s}(e)\oplus \bigoplus_{e\in\mathcal{E}\backslash\mathcal{E}_D}\mathbb{L}_2(e) : D^{2j}u\in D(\Gamma)\\
        &\hphantom{=\Bigg\{ u\in\bigcup_{s>2n-\frac{3}{2}}\bigoplus_{e\in\mathcal{E}_D}H^{s}(e)\oplus \bigoplus_{e\in\mathcal{E}\backslash\mathcal{E}_D}\mathbb{L}_2(e) :}\forall j\in\{0,\dots, n-1\}\Bigg\},    
    \end{split}
\end{equation*}
and set $K^0(\Gamma) = D^0(\Gamma)= \mathbb{L}_2(\Gamma)$.

\begin{lemma}\label{lem:H_dot_H_direct_1}
    Let $\alpha\in(0,\infty)$, we have $\dot{H}^{\alpha}\hookrightarrow \bigoplus_{e\in\mathcal{E}}H^{\alpha}(e)$.
\end{lemma}
\begin{proof}
    First let $\alpha\in\mathbb{N}$. The space $H^{2\alpha}(\Gamma)\cap K^{\alpha}(\Gamma)\cap D^{\alpha}(\Gamma)$ is dense in $\mathbb{L}_2(\Gamma)$ because it contains subspace spanned by the eigenfunctions $\{\varphi_i\}_{i\in\mathbb{N}}$. Observe that for $\kappa>0$ that the operator $L^{\alpha}:H^{2\alpha}(\Gamma)\cap K^{\alpha}(\Gamma)\cap D^{\alpha}(\Gamma)\to\mathbb{L}_2(\Gamma)$ is self-adjoint on $\mathbb{L}_2(\Gamma)$ and has eigenvalues $\{\lambda_i^{\alpha}\}_{i\in\mathbb{N}}$ and eigenfunctions $\{\varphi_i\}_{i\in\mathbb{N}}$. Thus $\dot{H}^{\alpha}$ is defined as the energetic space (\citeA[Example~4 in Section 5.8]{zeidler}) of the operator $L^{\alpha}$, which is the completion of $H^{2\alpha}(\Gamma)\cap K^{\alpha}(\Gamma)\cap D^{\alpha}(\Gamma)$ with respect to the norm $\|u\|_E=\langle L^{\alpha}u,u\rangle$. The norms $\|\cdot\|_{E}$ and $\|\cdot\|_{\bigoplus_{e\in\mathcal{E}}H^{\alpha}(e)}$ are equivalent on the space $H^{2\alpha}(\Gamma)\cap K^{\alpha}(\Gamma)\cap D^{\alpha}(\Gamma)$. The continuous inclusion follows because $\bigoplus_{e\in\mathcal{E}}H^{\alpha}(e)$ is closed under the norm $\|.\|_{\bigoplus_{e\in\mathcal{E}}H^{\alpha}(e)}$.\\
    Now assume that $\alpha\in(0,\infty)\backslash\mathbb{N}$. By \citeA[Theorem~4.36]{lunardi} we have $\dot{H}^{\alpha}\cong [\dot{H}^{\floor{\alpha}}, \dot{H}^{\ceil{\alpha}}]_{\alpha-\floor{\alpha}}$. By Lemma~\ref{lem:direct_sum_interpolation_lemma} and \citeA[Theorem~B.8]{mclean} we have $$\bigoplus_{e\in\mathcal{E}}H^{\alpha}=\left[\bigoplus_{e\in\mathcal{E}}H^{\floor{\alpha}}, \bigoplus_{e\in\mathbb{E}}H^{\ceil{\alpha}}\right]_{\alpha-\floor{\alpha}}.$$ Since $\dot{H}^{\floor{\alpha}}\hookrightarrow \bigoplus_{e\in\mathcal{E}}H^{\floor{\alpha}}(e)$ and $\dot{H}^{\ceil{\alpha}}\hookrightarrow \bigoplus_{e\in\mathcal{E}}H^{\ceil{\alpha}}(e)$, it follows by interpolation \cite[Theorem~B.2]{mclean} that $\dot{H}^{\alpha}\hookrightarrow\bigoplus_{e\in\mathcal{E}}H^{\alpha}(e)$.
\end{proof}

\begin{lemma}\label{lem:H_dot_H_Du}
    Let $\alpha\in(0,\infty)\backslash\{2k+\frac{1}{2}:k\in\mathbb{N}_0\}$ and $u\in\dot{H}^{\alpha}$. We have $u\in C^{\left(2\floor{\frac{\alpha}{2}-\frac{1}{4}}\right)}\cap K^{\floor{\frac{\alpha}{2}+\frac{1}{4}}}\cap D^{\floor{\frac{\alpha}{2}+\frac{3}{4}}}$ and
    \begin{equation}\label{eq:Du}
        D^j u = \sum_{i=1}^{\infty} \langle u,\varphi_i\rangle D^j\varphi_i = (-1)^{\floor{\frac{j}{2}}}\sum_{i=1}^{\infty} \widehat{\lambda}_i^{\floor{\frac{j}{2}}}\langle u,\varphi_i\rangle D^{j-2\floor{\frac{j}{2}}}\varphi_i
    \end{equation}
    for even $0\leq j\leq\alpha$ and odd $0\leq j<\alpha-\frac{1}{2}$, with the convergence of the series being absolutely uniform if $0\leq j<\alpha-\frac{1}{2}$. Furthermore,  $D^ju\in\dot{H}^{\alpha-j}$ for every even $0\leq j\leq\alpha$.
\end{lemma}

\begin{proof}
    Applying derivatives up to order $\floor{\alpha}$ is justified by Lemma~\ref{lem:H_dot_H_direct_1}. We will first show the result for $0\leq j<\alpha-\frac{1}{2}$. Using Theorem~\ref{thm:bounded_eigenfunctions}, Weyl's law, and Cauchy-Schwarz inequality, we have
\begin{equation}\label{eq:H_Du_eq_1}
    \begin{split}
        \sum_{i=i_b}^{\infty}|\langle u,\varphi_i\rangle||\varphi_i|&\leq \left(\sup_{i\in\mathbb{N}}\sup_{x\in \Gamma}|\varphi_i(x)|\right)\sum_{i=i_b}^{\infty}|\langle u,\varphi_i\rangle|\\&\leq \left(\sup_{i\in\mathbb{N}}\sup_{x\in \Gamma}|\varphi_i(x)|\right)\sqrt{\sum_{i=i_b}^{\infty}\lambda_i^{\alpha}\langle u,\varphi_i\rangle^2}\sqrt{\sum_{i=i_b}^{\infty}\lambda_i^{-\alpha}}<\infty.
    \end{split}
\end{equation}
By Weierstrass $M$-test we have that $\sum_{i=1}^{\infty}\langle u,\varphi_i\rangle\varphi_i$ converges uniformly to $u$. Let $e$ be an edge,  $\psi\in C^{\infty}_c(e)$, and $i_b = \min\{i\in\mathbb{N}: \widehat{\lambda}_i>0\}$. Applying Corollary~\ref{corr:eigenfunctions_derivative_bound} and by calculations similar to~\eqref{eq:H_Du_eq_1} we have
\begin{equation*}
    \begin{split}
        \sum_{i=i_b}^{\infty}|\langle u,\varphi_i\rangle|\int_e|D^j\varphi_i\psi|&\leq C\left(\sup_{x\in e}|\psi(x)|\right)\sum_{i=i_b}^{\infty}\widehat{\lambda}_i^{j/2}|\langle u,\varphi_i\rangle|\\&\leq C\left(\sup_{x\in e}|\psi(x)|\right)\sqrt{\sum_{i=i_b}^{\infty}\lambda_i^{\alpha}\langle u,\varphi_i\rangle^2}\sqrt{\sum_{i=i_b}^{\infty}\frac{\widehat{\lambda}_i^{j}}{\lambda_i^{\alpha}}}<\infty.
    \end{split}
\end{equation*}
Therefore $(-1)^j\sum_{i=1}^{\infty}\langle u,\varphi_i\rangle D^j\varphi_i$ converges uniformly to a limit. This limit must be $D^ju$ since by uniform convergence we have
\begin{equation*}
    \begin{split}
        \int_e uD^j\psi&=\sum_{i=1}^{\infty}\langle u,\varphi_i\rangle\int_e \varphi_iD^j\psi= (-1)^{j}\int_e\left(\sum_{i=1}^{\infty}\langle u,\varphi_i\rangle D^j\varphi_i\right)\psi.\\
    \end{split}
\end{equation*}
Furthermore, because the absolute uniform convergence preserves continuity and implies point-wise convergence, it follows that for every even $j<\alpha-\frac{1}{2}$ we have $D^ju\in C(\Gamma)\cap D(\Gamma)$, this is a consequence of $D^j\varphi_i,\varphi_i\in C(\Gamma)\cap D(\Gamma)$ for every $i\in\mathbb{N}$ and even $j\in\mathbb{N}$. The function $D^ju\in\dot{H}^{\alpha-j}$ belongs to $\dot{H}^{\alpha-j}$ because
\begin{equation*}
\begin{split}
    \sum_{i=1}^{\infty}\lambda_i^{\alpha-j}\langle u,D^j\varphi_i\rangle^2 = \sum_{i=1}^{\infty}\lambda_i^{\alpha-j}(\kappa^2+\hat{\lambda}_i^{j/2})^2\langle u,\varphi_i\rangle^2 < C\sum_{i=1}^{\infty}\lambda_i^{\alpha}\langle u,\varphi_i\rangle^2.
\end{split}
\end{equation*}
Analogous argument shows that
\begin{equation}\label{eq:partial_e_1}
    \partial_e w = \sum_{i=1}^{\infty}\langle w,\varphi_i\rangle\partial_e\varphi_i\quad\quad \forall w\in\bigcup_{s>\frac{3}{2}}\dot{H}^s.
\end{equation}
To show that $u\in K^{\floor{\frac{\alpha}{2}+\frac{1}{4}}}$, let $j<\alpha-\frac{3}{2}$ be a non-negative even integer and $v$ be a vertex. Applying \eqref{eq:partial_e_1} to $w=D^ju$ we have
$$\sum_{e\in\mathcal{E}_v}\partial_eD^{j}u = \sum_{i=1}^{\infty}(-\widehat{\lambda}_i)^{\frac{j}{2}}\langle u,\varphi_i\rangle \sum_{e\in\mathcal{E}_v}\partial_e \varphi_i = 0.$$
It remains to show that Equation \eqref{eq:Du} holds for
non-negative even integers $j$ in the interval $[\alpha-\frac{1}{2},\alpha]$. For $\alpha<2$ this reduces to $u=\sum_{i=1}^{\infty}\langle u,\varphi_i\rangle\varphi_i$. Assume $\alpha,j\geq2$ and let $w=D^{j-2}u\in\dot{H}^{\alpha-j+2}\subset \dot{H}^2$, by Lemma~\ref{lem:H_dot_H_direct_1} we have $\dot{H}^2\subset\bigoplus_{e\in\mathcal{E}}H^2(e)\cap \dot{H}^1$. Thus
\begin{equation*}
    \begin{split}
    D^ju &= (\kappa^2 - L)w = \kappa^2\sum_{i=1}^{\infty}\langle w,\varphi_i\rangle\varphi_i - \sum_{i=1}^{\infty}\lambda_i\langle w,\varphi_i\rangle\varphi_i\\
    &= \sum_{i=1}^{\infty}(-\widehat{\lambda}_i)\langle w,\varphi_i\rangle\varphi_i = \sum_{i=1}^{\infty}(-\widehat{\lambda}_i)^{\frac{j}{2}}\langle u,\varphi_i\rangle\varphi_i,      
    \end{split}
\end{equation*}
where the last equality follows from the symmetry of the operator $D^{j-2}$ on the space $\dot{H}^{\alpha-j+2}$. Finally $D^ju\in \dot{H}^{\alpha-j}$ since $L$ is a bounded operator from $\dot{H}^{\alpha-j+2}$ into $\dot{H}^{\alpha-j}$.
\end{proof}
In what follows, the set $\mathbb{A} = (0,\infty)\backslash\{k+\frac{1}{2}:k\in\mathbb{N}_0\}$ plays a central role in the statement of the main characterization results in this and the following section.
\begin{lemma}\label{lem:H_density_1}
    Let $\alpha\in\mathbb{A}$ and $k>\alpha$ an integer. Then $H^{2k}(\Gamma)\cap K^{k}(\Gamma)\cap D^{k}(\Gamma)$ is dense in $H^{\alpha}(\Gamma)\cap K^{\floor{\frac{\alpha}{2}+\frac{1}{4}}}(\Gamma)\cap D^{\floor{\frac{\alpha}{2}+\frac{3}{4}}}(\Gamma)$.
\end{lemma}
\begin{proof}
    Let $u\in H^{\alpha}(\Gamma)\cap K^{\floor{\frac{\alpha}{2}+\frac{1}{4}}}(\Gamma)\cap D^{\floor{\frac{\alpha}{2}+\frac{3}{4}}}(\Gamma)$. Define $P:\Gamma\to\mathbb{R}$ as a polynomial on each edge $e$ that satisfies the constraints $\partial_e^iP_e(v)=\partial_e^i u_e(v)$ for every vertex $v\in e$ and $0\leq i< 2\floor{\frac{\alpha}{2}+\frac{3}{4}}$ and $\partial_e^iP_e(v)=0$ for every $2\floor{\frac{\alpha}{2}+\frac{3}{4}}\leq i\leq 2k$. With this construction we have $u-P\in\bigoplus_{e\in\mathcal{E}} H^{\alpha}_{0}$. Defining $v_n\in\bigoplus_{e\in\mathcal{E}} C^{\infty}_{c}$ such that $v_n\to u-P$ in $H^{\alpha}(\Gamma)$, the result follows because $v_n+P\in H^{2k}(\Gamma)\cap K^k(\Gamma)\cap D^k(\Gamma)$.
\end{proof}

\begin{lemma}\label{lem:H_alpha_H_dot_inclusion}
    For $\alpha\in\mathbb{N}$, we have
    $$H^{2\alpha}(\Gamma)\cap K^{\alpha}(\Gamma)\cap D^{\alpha}(\Gamma)\subset \dot{H}^{\alpha}.$$
\end{lemma}
\begin{proof}
    For $\alpha=1$, the statement follows from the definition of $\dot{H}^1$. For $\alpha\geq 2$, first note that
    \begin{equation}\label{eq:H_alpha_1_claim_1}
      \left\{f\in H^2(\Gamma)\cap K(\Gamma)\cap D(\Gamma),\;\Delta_{\Gamma} f\in\dot{H}^{\alpha-2}\right\} \Rightarrow f\in\dot{H}^{\alpha}.  
    \end{equation}
    This follows from the calculations
    \begin{equation*}
        \begin{split}
            \sum_{i=1}^{\infty}\lambda_i^{\alpha}\langle f,\varphi_i\rangle &= \sum_{i=1}^{\infty}\lambda_{i-1}^{\alpha}\langle f,L\varphi_i\rangle^2=\sum_{i=1}^{\infty}\lambda_{i}^{\alpha-1}\langle Lf,\varphi_i\rangle^2\\
            &\leq2\kappa^2\sum_{i=1}^{\infty}\lambda_{i}^{\alpha-2}\langle f,\varphi_i\rangle^2 + 2\sum_{i=1}^{\infty}\lambda_{i}^{\alpha-2}\langle \Delta_{\Gamma}f,\varphi_i\rangle^2<\infty.
        \end{split}
    \end{equation*}
    Also, by definition, we have
    \begin{equation}\label{eq:H_alpha_1_claim_2}
        \begin{split}
            f\in H^{2\alpha}(\Gamma)\cap K^{\alpha}(\Gamma)\cap D^{\alpha}(\Gamma)&\Rightarrow \Big\{f\in H^2(\Gamma)\cap K(\Gamma)\cap D(\Gamma),\\& \mathrel{\hphantom{\Rightarrow}}\Delta_{\Gamma} f\in H^{2\alpha-2}(\Gamma)\cap K^{\alpha-1}(\Gamma)\cap D^{\alpha-1}(\Gamma)\Big\}.
        \end{split}
    \end{equation}Now, we use induction as follows. Suppose that $H^{2\alpha}(\Gamma)\cap K^{\alpha}(\Gamma)\cap D^{\alpha}(\Gamma)\subset \dot{H}^{\alpha}$ for $\alpha\geq1$ and let $u\in H^{2\alpha+2}(\Gamma)\cap K^{\alpha+1}(\Gamma)\cap D^{\alpha+1}(\Gamma)$. By \eqref{eq:H_alpha_1_claim_2} and the induction hypothesis we have that $u\in H^{2}(\Gamma)\cap K(\Gamma)\cap D(\Gamma)$, and $\Delta_{\Gamma} u\in H^{2\alpha}(\Gamma)\cap K^{\alpha}(\Gamma)\cap D^{\alpha}(\Gamma)\subset\dot{H}^{\alpha}$. Then, by \eqref{eq:H_alpha_1_claim_1} we have $u\in\dot{H}^{\alpha+1}$.
\end{proof}
The following identification is the cornerstone of our framework. It shows that the energetic spaces of the operator $L^{\alpha}=(\kappa^2-\Delta)^{\alpha}$ coincide exactly with the intrinsic Sobolev spaces we have constructed coupled with appropriate vertex conditions, thereby justifying their use for PDE analysis on metric graphs.
\begin{theorem}\label{thm:main_1}
    For $\alpha\in\mathbb{A}$, we have
    $$\dot{H}^{\alpha}\cong H^{\alpha}(\Gamma)\cap K^{\floor{\frac{\alpha}{2}+\frac{1}{4}}}(\Gamma)\cap D^{\floor{\frac{\alpha}{2}+\frac{3}{4}}}(\Gamma).$$
\end{theorem}
\begin{proof}
We prove the statement by considering six different cases.\hfill\\
\textbf{Case 1:} $\alpha\in\mathbb{N}$. Let $k\geq\alpha+1$ be an integer. By Lemma~\ref{lem:H_alpha_H_dot_inclusion} we have $$H^{2k}(\Gamma)\cap K^{k}(\Gamma)\cap D^{k}(\Gamma)\subset \dot{H}^{\alpha+1}\subset \dot{H}^{\alpha}.$$ Moreover, Lemmas \ref{lem:H_dot_H_direct_1} and \ref{lem:H_dot_H_Du} imply the inclusion
\begin{equation}\label{eq:H_H_dot_iclusion_1}
    H^{2k}(\Gamma)\cap K^{k}(\Gamma)\cap D^{k}(\Gamma)\subset \dot{H}^{\alpha}\subset H^{\alpha}(\Gamma)\cap K^{\floor{\frac{\alpha}{2}+\frac{1}{4}}}(\Gamma)\cap D^{\floor{\frac{\alpha}{2}+\frac{3}{4}}}(\Gamma).
\end{equation}
    By Lemma~\ref{lem:H_density_1} the subpace $H^{2k}(\Gamma)\cap K^{k}(\Gamma)\cap D^{k}(\Gamma)$ is dense in $H^{\alpha}(\Gamma)\cap K^{\floor{\frac{\alpha}{2}+\frac{1}{4}}}(\Gamma)\cap D^{\floor{\frac{\alpha}{2}+\frac{3}{4}}}(\Gamma)$. To conclude equality from \eqref{eq:H_H_dot_iclusion_1} it suffices to show that $\dot{H}^{\alpha}$ is closed in $\bigoplus_{e\in\mathcal{E}}H^{\alpha}(e)$. Indeed, Indeed, $\dot{H}^{\alpha}$  is the energetic space associated with the operator $L^{\alpha}$, which is complete under the norm $\|u\|_E^2=\langle L^{\alpha}u,u\rangle$. However, by integration by parts we have the estimate $\|u\|_E\leq C\|u\|_{H^{\alpha}(\Gamma)}$. The closedness under $\|u\|_{H^{\alpha}(\Gamma)}$ follow immediately.\\
    
    \noindent\textbf{Case 2:} $\alpha\in(0,\frac{1}{2})$. By \citeA[Theorem~4.36]{lunardi} we have $\dot{H}^{\alpha}=[\mathbb{L}_2(\Gamma), \dot{H}^1]_{\alpha}$.
    Since $\dot{H}^1(\Gamma)=H^1(\Gamma)\cap D(\Gamma)$ by Case~1, and since $\bigoplus_{e\in\mathcal{E}}H^1_0\subset H^1(\Gamma)\cap D(\Gamma)$ we have
    \begin{equation*}
        \begin{split}
        [\mathbb{L}_2(\Gamma), \dot{H}^1]_{\alpha}
        = [\mathbb{L}_2(\Gamma), H^1(\Gamma)\cap D(\Gamma)]_{\alpha}&\supset\left[\bigoplus_{e\in\mathcal{E}}\mathbb{L}_2(e),\bigoplus_{e\in\mathcal{E}}H^1_0(e)\right]_{\alpha}\\
            &= \bigoplus_{e\in\mathcal{E}}H^{\alpha}(e),
        \end{split}
    \end{equation*}
    the last line follows from Lemma~\ref{lem:direct_sum_interpolation_lemma} and \citeA[Theorem~3.40]{mclean}. Therefore $\dot{H}^{\alpha}\cong \bigoplus_{e\in\mathcal{E}}H^{\alpha}(e)$ by Lemma~\ref{lem:H_dot_H_direct_1}. Finally, by Theorems~\ref{thm:W_chatacterization_3} and \ref{thm:H_W_equivalence} we have $H^{1}(\Gamma)\cong\bigoplus_{e\in\mathcal{E}}H^{\alpha}(e)$.\\
    
    \noindent\textbf{Case 3:} $\alpha\in(\frac{1}{2},1)$. We proceed as in Case~2, utilizing \citeA[Theorem~B.8]{mclean}
    \begin{equation*}
        \begin{split}
        \dot{H}^{\alpha}=[\mathbb{L}_2(\Gamma), \dot{H}^1]_{\alpha}
        &= \left[\bigoplus_{e\in\mathcal{E}}\mathbb{L}_2(e), \bigoplus_{e\in\mathcal{E}}H^1(e)\cap C(\Gamma)\cap D(\Gamma)\right]_{\alpha}\\&\supset\left[\bigoplus_{e\in\mathcal{E}}\mathbb{L}_2(e),\bigoplus_{e\in\mathcal{E}}H^1(e)\right]_{\alpha}\cap C(\Gamma)\cap D(\Gamma)\\
        &= \bigoplus_{e\in\mathcal{E}}H^{\alpha}(e)\cap C(\Gamma)\cap D(\Gamma)=H^{\alpha}(\Gamma)\cap D(\Gamma).
        \end{split}
    \end{equation*}
    Lemmas~\ref{lem:H_dot_H_Du} and~\ref{lem:H_dot_H_direct_1} imply $\dot{H}^{\alpha}\cong H^{\alpha}(\Gamma)\cap D(\Gamma)$.\\
    
    \noindent\textbf{Case 4:} $\alpha\in(1,\frac{3}{2})$. Let $u,v\in H^2(\Gamma)\cap K(\Gamma)\cap D(\Gamma)\cong \dot{H}^2$. Let $v=L^{1-\alpha/2}v_1$ for some $v_1\in\dot{H}^{4-\alpha}\subset \dot{H}^2\cong H^1(\Gamma)\cap D(\Gamma)$. We have
    \begin{equation*}
        \begin{split}
            \langle L^{\alpha/2}u,v\rangle=\langle u,L^{\alpha/2} v\rangle = \langle u, v_1\rangle = \kappa^2\langle u,v_1\rangle + \sum_{e\in\mathcal{E}}\langle Du,Dv_1\rangle_e.
        \end{split}
    \end{equation*}
    We estimate
    \begin{equation*}
        \begin{split}
            \sum_{e\in\mathcal{E}}\langle Du,Dv_1\rangle_e & \leq\sum_{e\in\mathcal{E}} \|Du\|_{H^{\alpha-1}(e)}\|Dv_1\|_{H^{1-\alpha}(e)}\\
            &\leq C\sum_{e\in\mathcal{E}}\|u\|_{H^{\alpha}(e)}\|v_1\|_{H^{2-\alpha}(e)}\\
            &\leq C\|u\|_{H^{\alpha}(\Gamma)}\|v_1\|_{\dot{H}^{2-\alpha}(\Gamma)}\leq C\|u\|_{H^{\alpha}(\Gamma)}\|v\|_{\mathbb{L}_2(\Gamma)},
        \end{split}
    \end{equation*}
    where we used the fact that $H_0^{\alpha-1}(e)=H^{\alpha-1}(e)$, $H^{2-\alpha}(\Gamma)\cong \dot{H}^{2-\alpha}$ by Case~2, and that $L^{1-\alpha/2}$ is continuous from $\dot{H}^{2-\alpha}$ onto $\mathbb{L}_2(\Gamma)$. We also have $\kappa^2\langle u,v\rangle\leq \kappa^2\|u\|_{\mathbb{L}_2(\Gamma)}\|v\|_{\mathbb{L}_2(\Gamma)}\leq C\|u\|_{\mathbb{L}_2(\Gamma)}\|v\|_{H^{\alpha}(\Gamma)}$. Therefore,
    \begin{equation}\label{eq:H_dot_H_alpha_norm_comparison}
        \begin{split}
            \|u\|_{\dot{H}^{\alpha}} = \|L^{\alpha/2}u\|_{\mathbb{L}_2(\Gamma)}
            =\sup_{\|v\|_{\mathbb{L}_2(\Gamma)}=1}\langle L^{\alpha/2}u,v \rangle\leq C\|u\|_{H^{\alpha}(\Gamma)}
        \end{split}
    \end{equation}
    for every $u\in H^{\alpha}(\Gamma)\cap K(\Gamma)\cap D(\Gamma)\cong \dot{H}^{\alpha}$. Now, Let $u\in H^{\alpha}(\Gamma)\cap D(\Gamma)$, using Lemma~\ref{lem:H_density_1} let $\{u_n\}_{n\in\mathbb{N}}$ be a sequence such that $u_n\in H^{2}(\Gamma)\cap K(\Gamma)\cap D(\Gamma)$ and $u_n\to u$ in $H^{\alpha}(\Gamma)$. By \eqref{eq:H_dot_H_alpha_norm_comparison} the sequence $\{u_n\}_{n\in\mathbb{N}}$ is Cauchy in $\dot{H}^{\alpha}$, hence by completeness of $\dot{H}^{\alpha}$ the sequence $\{u_n\}$ converges to a limit in $\dot{H}^{\alpha}$, this limit must be $u$ since both convergences in $H^{\alpha}(\Gamma)$ and $\dot{H}^{\alpha}$ imply $\mathbb{L}_2(\Gamma)$ convergence. Consequently, we have the estimate 
    $$\|u\|_{\dot{H}^{\alpha}}=\lim_{n\to\infty}\|u_n\|_{\dot{H}^{\alpha}}\leq C\lim_{n\to\infty}\|u_n\|_{H^{\alpha}(\Gamma)} = C\|u\|_{H^{\alpha}(\Gamma)}.$$
    Therefore $H^{\alpha}(\Gamma)\cap D(\Gamma)\hookrightarrow \dot{H}^{\alpha}$. Lemmas \ref{lem:H_dot_H_direct_1} and \ref{lem:H_dot_H_Du} imply that $\dot{H}^{\alpha}\hookrightarrow H^{\alpha}(\Gamma)\cap D(\Gamma)$, thus $\dot{H}^{\alpha}\cong H^{\alpha}(\Gamma)\cap D(\Gamma)$.\\
    
    \noindent\textbf{Case 5}: $\alpha\in(\frac{3}{2},2)$. Analogously to Case~1, we have
    $$H^2(\Gamma)\cap K(\Gamma)\cap D(\Gamma)\subset\dot{H}^{\alpha}\subset H^{\alpha}(\Gamma)\cap K(\Gamma)\cap D(\Gamma).$$
    Then, arguing analogously to Case 3 yields the desired identification.\\
    
    \noindent\textbf{Case 6:} $\alpha\geq2$. We proceed by induction on the sets of the form $$S_n=[n,n+2]\backslash\{n+\frac{1}{2},n+\frac{3}{2}\}$$ for $n\in\mathbb{N}_0$. The previous cases are the base cases $n=0$. Let $u\in H^{\alpha}(\Gamma)\cap K^{\floor{\frac{\alpha}{2}+\frac{1}{4}}}(\Gamma)\cap D^{\floor{\frac{\alpha}{2}+\frac{3}{4}}}(\Gamma)$ for $\alpha\in S_n$, and assume that $H^{\beta}(\Gamma)\cap K^{\floor{\frac{\beta}{2}+\frac{1}{4}}}(\Gamma)\cap D^{\floor{\frac{\beta}{2}+\frac{3}{4}}}(\Gamma)\cong \dot{H}^{\beta}$ for every $\beta\in\bigcup_{0\leq m<n}S_m$. Thus $u\in \dot{H}^{2}(\Gamma)$ and $Lu\in\dot{H}^{\alpha-2}$. By Claim~1 in the proof of Lemma~\ref{lem:H_alpha_H_dot_inclusion} we have $u\in\dot{H}^{\alpha}$. Thus $H^{\alpha}(\Gamma)\cap K^{\floor{\frac{\alpha}{2}+\frac{1}{4}}}(\Gamma)\cap D^{\floor{\frac{\alpha}{2}+\frac{3}{4}}}(\Gamma)\subset \dot{H}^{\alpha}$. Lemmas~\ref{lem:H_dot_H_Du} and~\ref{lem:H_dot_H_direct_1} imply $H^{\alpha}(\Gamma)\cap K^{\floor{\frac{\alpha}{2}+\frac{1}{4}}}(\Gamma)\cap D^{\floor{\frac{\alpha}{2}+\frac{3}{4}}}(\Gamma)\cong \dot{H}^{\alpha}$, this completes the induction step.
\end{proof}
For $\beta>0$, define the Sobolev associated to the vertex conditions $K(\Gamma)$ and $D(\Gamma)$ 
$$H_{K,D}^{\beta}(\Gamma):= H^{\beta}(\Gamma)\cap K^{\floor{\frac{\beta}{2}+\frac{1}{4}}}(\Gamma)\cap D^{\floor{\frac{\beta}{2}+\frac{3}{4}}}(\Gamma),$$
for $\beta<0$, define the negative order Sobolev space associated to the vertex conditions $K(\Gamma)$ and $D(\Gamma)$ as the dual space $$H_{K,D}^{\beta}(\Gamma):=\left(H_{K,D}^{-\beta}(\Gamma)\right)^*.$$
If $\mathcal{V}_D=\varnothing$, we denote $H_{K}^{\beta}(\Gamma):=H_{K,D}^{\beta}(\Gamma)$, and if $\mathcal{V}_K:=\varnothing$, denote $H_{D}^{\beta}(\Gamma)=H_{K,D}^{\beta}(\Gamma)$.

Recall that the operator $L^{\alpha/2}$ maps $\dot{H}^{\beta}$ onto $\dot{H}^{\alpha-\beta}$ isometrically, combining this with the identification established in Theorem~\ref{thm:main_1}, we obtain the following corollary.
\begin{corr}\label{corr:existence_L_1}
    Let $\alpha>0$ and $\beta\in\mathbb{R}\backslash\{n+\frac{1}{2}:n\in\mathbb{Z}\}$, and suppose that $2(\alpha+\beta)\not\in\mathbb{Z}$. Let 
    $$f\in H^{\beta}_{K,D}(\Gamma)\cong \dot{H}^{\beta}.$$
    The fractional PDE 
    $$L^{\alpha/2}u=f$$
    admits a unique solution 
    $$u\in H^{\beta+\alpha}_{K,D}(\Gamma)\cong \dot{H}^{\alpha+\beta}.$$
    Moreover, there exists a constant $C$ which depends on $\alpha$ and $\beta$ only that satisfies
    $$\|u\|_{H^{\alpha+\beta}_{K,D}(\Gamma)}\leq C \|f\|_{H^{\beta}_{K,D}(\Gamma)}.$$
\end{corr}
Another application of Theorem~\ref{thm:main_1} is the well-posedness of the SPDE $L^{\alpha/2}u=\mathcal{W}$ where $\mathcal{W}$ is a Gaussian white noise. We state the result as the following corollary.
\begin{corr}\label{corr:existence_L_1_W}
    Let $\alpha,\epsilon>0$, and $\mathcal{W}$ be a Gaussian white noise defined on $\Gamma$. The SPDE
    $$L^{\alpha/2}u=\mathcal{W}$$
    has almost surely a unique solution
    $u\in H^{\alpha-\frac{1}{2}-\epsilon}_{K,D}(\Gamma)$.
\end{corr}
\begin{proof}
    Follows from Corollary~\ref{corr:existence_L_1} and Theorem~\ref{thm:main_1} by noting that $\mathcal{W}\in\dot{H}^{-\frac{1}{2}-\epsilon}$ (see \citeA{bern_gaussian_matern}).
\end{proof}

\section{Application to more general operators}
\label{sec:appli_2}
Consider the operator $\mathcal{L}= \kappa^2 - \nabla(a\nabla)$ where the coefficients $a,\kappa\in \mathbb{L}_{\infty}(\Gamma)$ satisfy $$\inf_{x\in\Gamma}\kappa(x)>0,\quad\inf_{x\in\Gamma}a(x)>0.$$ 
The operator $\mathcal{L}$ is defined from $H^1_{D}(\Gamma)=H^1_{K,D}(\Gamma)$ into the dual space $H^{-1}_{D}(\Gamma)=H^{-1}_{K,D}(\Gamma)$ via the bilinear form
$$\langle \mathcal{L} u,v\rangle = \langle \kappa^2 u,v\rangle + \langle a\nabla u,\nabla v\rangle\quad \forall u,v\in H^1_D(\Gamma).$$
Note that the quantity $\langle a\nabla u,\nabla v\rangle = \langle aD_{\eta} u,D_{\eta} v\rangle$ is independent of parameterizations $\eta$. Let $L=1-\Delta_{\Gamma}$, and let $C>0$ such that $a,\kappa^2,\frac{1}{a},\frac{1}{\kappa^2} <C$. For every $u\in H^1_D(\Gamma)$, we have the inequality
\begin{equation}\label{eq:L_L_inequality}
\begin{split}
    \langle  \mathcal{L}u,u\rangle = \|\kappa u\|^2_{\mathbb{L}_2(\Gamma)} + \|\sqrt{a}\nabla u\|^2_{\mathbb{L}_2(\Gamma)} &\geq \frac{1}{C}\|u\|^2_{\mathbb{L}_2(\Gamma)} + \frac{1}{C}\|\nabla u\|^2_{\mathbb{L}_2(\Gamma)}\\
    &= \frac{1}{C}\langle Lu,u\rangle.
\end{split}
\end{equation}
and by the Cauchy-Schwarz inequality
\begin{equation}\label{eq:L_L_inequality_2}
    \langle  \mathcal{L}u,v\rangle \leq C\|u\|_{H^1(\Gamma)}\|v\|_{H^1(\Gamma)}.
\end{equation}
For every $u,v\in H^{1}_D(\Gamma)$. Therefore $(u,v)\mapsto \langle \mathcal{L}u,v\rangle$ defines a continuous coersive bilinear form on $H^1_D(\Gamma)$. Direct application of the Lax-Milgram theorem yields the following theorem.
\begin{theorem}\label{thm:existence_calL_1}
    Let $f\in H^{-1}(\Gamma)$. The PDE
    $$\mathcal{L}u=f$$
    has a unique solution $u\in H^{1}_D(\Gamma)$. Moreover, there exists a constant $C$, independent of $f$ and $u$, that satisfies
    $$\|u\|_{H^1(\Gamma)}\leq C\|f\|_{H^{-1}_D(\Gamma)}.$$
\end{theorem}
To establish higher regularity results we assume the function $a$ is Lipschitz, then $\mathcal{L}$ can be defined on $H^2(\Gamma)$, and by \citeA[Theorem~2.7]{mathcomp_paper} the restriction $$\mathcal{L}:H^2(\Gamma)\cap K(\Gamma)\cap D(\Gamma)\to\mathbb{L}_2(\Gamma)$$ is a self-adjoint and positive operator with a compact inverse. Consequently, we may define the fractional powers $\mathcal{L}^{\alpha/2}$ and spaces $\dot{H}^{\alpha}_{\mathcal{L}} = \mathcal{D}(\mathcal{L}^{\alpha/2})$ for $\alpha\in\mathbb{R}$, analogously to the construction of $L^{\alpha/2}$ and $\dot{H}^{\alpha}$ for the operator $L$.

\begin{remark}
    When $\mathcal{V}_D\not=\varnothing$ the condition $\inf_{x\in\Gamma}\kappa(x)>0$ can be dropped. This is because, in this case, the following Poincar\'e inequality holds
    \begin{equation*}
        \|u\|_{\mathbb{L}_2(\Gamma)}\leq C\|Du\|_{\mathbb{L}_2(\Gamma)}.
    \end{equation*}
    We formalize this result as the following lemma.
\end{remark}
\begin{lemma}[Poincar\'e Inequality]
    Assume that $\mathcal{V}_D\not=\varnothing$, and let $u\in H^1_D(\Gamma)$. We have
    \begin{equation*}
        \|u\|_{\mathbb{L}_2(\Gamma)}\leq C\|Du\|_{\mathbb{L}_2(\Gamma)}.
    \end{equation*}
\end{lemma}
\begin{proof}
    Let $x_0\in\mathcal{V}_D$, $e\in\mathcal{E}$, $(y,z)\in \partial e\times e$, and $\zeta$ a parameterization of $e$ such that $\zeta(0)=y$. We have
    \begin{equation*}
        |u(z) - u(y)| = \left|\int_0^{\zeta^{-1}(z)} D(u\circ\zeta)(t)dt\right| \leq \int_0^{l_e}|D(u\circ\zeta)(t)|dt \leq \|Du\|_{\mathbb{L}_1(\Gamma)}.
    \end{equation*}
    Now let $y=y_0,y_1,\cdots,y_k=x_0$ be a sequence of vertices in $\Gamma$ that constitutes a shortest path from $y$ to $x_0$. By the previous inequality, we have
    \begin{equation*}
    \begin{split}
        |u(z)| = |u(z) - u(x_0)|&\leq |u(z)-u(y)| + \sum_{i=1}^{k}|u(y_i) - u(y_{i-1})|\\
        &\leq k\|Du\|_{\mathbb{L}_1(\Gamma)}\leq \#\mathcal{V}\|Du\|_{\mathbb{L}_1(\Gamma)}
    \end{split}
    \end{equation*}
    Finally,
    \begin{equation*}
        \|u\|_{\mathbb{L}_2(\Gamma)}^2 = \int_{\Gamma}|u(z)|^2dz\leq \#\mathcal{V} \|Du\|_{\mathbb{L}_1(\Gamma)}^2\int_{\Gamma}dz\leq C\|Du\|_{\mathbb{L}_2(\Gamma)}^2.
    \end{equation*}
\end{proof}

\begin{theorem}\label{thm:main_iden_2}
    Assume that $a$ is Lipschitz and $\alpha\in[0,2]$. Then $\dot{H}^{\alpha}_{\mathcal{L}}\cong \dot{H}^{\alpha}$.
\end{theorem}
\begin{proof}
    By \eqref{eq:L_L_inequality} and~\eqref{eq:L_L_inequality_2} it follows that $\dot{H}^1_{\mathcal{L}}\cong \dot{H}^1$. Consequently by \citeA[Theorem~4.36]{lunardi} we have $\dot{H}^{\alpha}_{\mathcal{L}}\cong \dot{H}^{\alpha}$ for all $\alpha\in(0,1)$. Next, we prove that $\dot{H}^{2}_{\mathcal{L}}\cong\dot{H}^{2}$, since applying \citeA[Theorem~4.36]{lunardi} again implies $\dot{H}^{\alpha}_{\mathcal{L}}\cong\dot{H}^{\alpha}$ for $\alpha\in(1,2)$. First, the operator $$\mathcal{L}:H^{2}(\Gamma)\cap K(\Gamma)\cap D(\Gamma)\to\mathbb{L}_2(\Gamma)$$ is bijective. Indeed, by Theorem~\ref{thm:existence_calL_1} the equation $\mathcal{L}u=f\in\mathbb{L}_2(\Gamma)$ has a unique solution $u\in H^1(\Gamma)\cap D(\Gamma)$. 
    Let $P$ be an edge-wise polynomial in $H^2(\Gamma)$ such that $v=u-P\in\bigoplus_{e\in\mathcal{E}}H^{1}_0(e)$. The function $v$ is a weak solution of 
    \begin{equation}\label{eq:v_weak_1}
        \mathcal{L}v=f-\mathcal{L}P\in\mathbb{L}_2(\Gamma).
    \end{equation}
    Applying \citeA[Theorem~9.15]{gilbarg_trudinger} to the function $v$, we have $v=u-P\in\bigoplus_{e\in\mathcal{E}}H^2(e)$. Since $u\in D(\Gamma)$ we obtain that $u\in H^2(\Gamma)\cap D(\Gamma)$. 
    To show that $u\in K(\Gamma)$, let $v\in\mathcal{V}_K$ and choose $w_n\in H^2(\Gamma)\cap D(\Gamma)$ satisfying $w_n(v)\to1$, and $w_n(x)=0$ for every $x\in\mathcal{V}\backslash\{v\}$, and $\|w_n\|_{\mathbb{L}_2(\Gamma)}\to 0$. Then we have 
    \begin{equation*}
        \begin{split}
            0=\lim_{n\to\infty} \langle f,w_n\rangle &= \lim_{n\to\infty} \left(\langle \kappa^2 u,w_n\rangle - \langle \nabla (a\nabla u), w_n\rangle + \sum_{e\in\mathcal{N}_v}\partial_eu(v)w_n(v)\right) \\&= \sum_{e\in\mathcal{N}_v}\partial_eu(v).
        \end{split}
    \end{equation*}
    Therefore $u\in H^{2}(\Gamma)\cap K(\Gamma)\cap D(\Gamma)$.
    On the other hand, $\mathcal{L}$ maps $\dot{H}^2_{\mathcal{L}}$ bijectively and continuously onto $\dot{H}^0_{\mathcal{L}}\cong\mathbb{L}_2(\Gamma)$. Since $H^2(\Gamma)\cap K(\Gamma)\cap D(\Gamma)\subset\dot{H}^2_{\mathcal{L}}$ we must have 
    $$\dot{H}^2_{\mathcal{L}} = H^2(\Gamma)\cap K(\Gamma)\cap D(\Gamma)\cong \dot{H}^2,$$ 
    by Theorem~\ref{thm:W_chatacterization_3}. The map $\mathcal{L}$ is continuous, and the open mapping theorem implies that $\mathcal{L}^{-1}:\mathbb{L}_2(\Gamma)\to\dot{H}^2$ is continuous, hence $$\frac{1}{C}\|u\|_{\dot{H}^2}\leq \|\mathcal{L}u\|_{\mathbb{L}_2(\Gamma)}\leq C \|u\|_{\dot{H}^2},$$
    for a constant $C>0$ independent of $u$. Since $\|\mathcal{L}u\|_{\mathbb{L}_2(\Gamma)} = \|u\|_{\dot{H}^{2}_{\mathcal{L}}}$ we have $\dot{H}^2_{\mathcal{L}}\cong \dot{H}^2$.
\end{proof}
Similarly to Corollaries~\ref{corr:existence_L_1} and~\ref{corr:existence_L_1_W},  Theorem~\ref{thm:main_iden_2} and Theorem~\ref{thm:main_1} imply the following corollaries.
\begin{corr}\label{corr:well_posedness_calL_2}
    Assume that the function $a$ is Lipschitz, let $(\alpha,\beta)\in(-4,4)\times\left((-2,2)\backslash\{\pm\frac{1}{2},\pm\frac{3}{2}\}\right)$ such that $\alpha+\beta\in(-2,2)\backslash\{\pm\frac{1}{2},\frac{3}{2}\}$, and let 
    $f\in H^{\beta}_{K,D}(\Gamma)$. The fractional PDE
    $$\mathcal{L}^{\alpha/2}u=f,$$
    has a unique solution $u\in H^{\beta+\alpha}_{K,D}(\Gamma)$.
    Furthermore, there exists a constant $C$ which depends on $\alpha$ and $\beta$ only that satisfies
    $$\|u\|_{H^{\alpha+\beta}_{K,D}(\Gamma)}\leq C \|f\|_{H^{\beta}_{K,D}(\Gamma)}.$$
\end{corr}
\begin{corr}\label{corr:well_posedness_calL_2_W}
    Assume $a$ is Lipschitz, let $\alpha\in(-\frac{3}{2},\frac{5}{2})$, $\epsilon>0$, and $\mathcal{W}$ be a Gaussian white noise defined on $\Gamma$. The SPDE
    $$\mathcal{L}^{\alpha/2}u=\mathcal{W},$$
    has almost surely a unique solution
    $u\in H^{\alpha-\frac{1}{2}-\epsilon}_{K,D}(\Gamma)$.
\end{corr}

To pursue further identifications, we follow the method in \citeA{markov_paper} to prove Theorem~\ref{thm:ident_3} below. 
From here on, let $\alpha\in\mathbb{A}$ be a fixed real number.
\begin{assum}\label{assum:regular_coeff}
    If $\alpha\in\mathbb{A}\cap(2,\infty)$, assume that $a\in\bigoplus_{e\in\mathcal{E}}C^{\ceil{\alpha}-2,1}(e)$ and $\kappa\in\bigoplus_{e\in\mathcal{E}}C^{\ceil{\alpha}-3,1}(e)$.
\end{assum}
\begin{lemma}\label{lem:H_ladder_1}
    Assume \eqref{assum:regular_coeff} and $\alpha\in\mathbb{A}\cap[2,\infty)$. We have
    $$f\in \bigoplus_{e\in\mathcal{E}}H^{\alpha}(e)\Longleftrightarrow f\in\bigoplus_{e\in\mathcal{E}}H^{2}(e),\;\mathcal{L}f\in \bigoplus_{e\in\mathcal{E}}H^{\alpha-2}(e).$$ 
\end{lemma}
\begin{proof}
The direction $\Rightarrow$ is clear. For the direction $\Leftarrow$, suppose that $f\in\bigoplus_{e\in\mathcal{E}}H^{k}(e)$ for a real number $k\geq2$. We have 
$$\Delta f = \frac{1}{a}\left(-\mathcal{L}f + \kappa^2f - (\nabla a)\nabla f\right)\in \bigoplus_{e\in\mathcal{E}}H^{\min\{\alpha-2,k-1\}}(e)$$ 
Since $f,\nabla f\in\mathbb{L}_2(\Gamma)$ we have $f\in\bigoplus_{e\in\mathcal{E}}H^{\min\{\alpha,k+1\}}(e)$. Starting with $k=2$, the recursion $k\mapsto\min\{\alpha,k+1\}$ increases and stabilizes at $\alpha$, thus we have $f\in\bigoplus_{e\in\mathcal{E}}H^{\alpha}(e)$.
\end{proof}

\begin{theorem}\label{thm:ident_3}
    Assume \eqref{assum:regular_coeff} and $\alpha\in\mathbb{A}\cap[1,\infty)$. We have
    \begin{equation}\label{eq:H_calL_main}
    \begin{split}
        \dot{H}^{\alpha}_{\mathcal{L}}\cong \Bigg\{&f\in\bigoplus_{e\in\mathcal{E}}H^{\alpha}(e): \mathcal{L}^{\floor{\frac{\alpha}{2}-\frac{1}{4}}}f\in C(\Gamma)\cap D(\Gamma),\\& \forall m\in\left\{0,\dots,\floor{\frac{\alpha}{2}-\frac{3}{4}}\right\}, \mathcal{L}^mf\in C(\Gamma)\cap D(\Gamma)\cap K(\Gamma)  \Bigg\}.
    \end{split}
    \end{equation} 
\end{theorem}
\begin{proof}We prove the theorem in four steps. Step~1 establishes the equivalence of the norms $\|.\|_{\dot{H}_{\mathcal{L}}^{\alpha}}$ and $\|.\|_{\bigoplus_{e\in\mathcal{E}}H^{\alpha}(e)}$ on the space $\dot{H}^{\alpha}_{\mathcal{L}}$, Step~2 provides a general characterization of $\dot{H}^{\alpha}_{\mathcal{L}}$, and in Step~3 and~4 we use the characterization from Step 2 to prove the theorem’s statement, first for $\alpha\in[2,5/2)$, then for general $\alpha$.\hfill\\

\noindent\textbf{Step 1:} By Theorem~\ref{thm:main_iden_2}, and Theorem~\ref{thm:main_1} we have $\dot{H}^{\alpha}_{\mathcal{L}}\subset\bigoplus_{e\in\mathcal{E}}H^{\alpha}(e)$ for $\alpha\in[0,2]$, applying Lemma~\ref{lem:H_ladder_1} inductively yields $$\dot{H}^{\alpha}_{\mathcal{L}}\subset\bigoplus_{e\in\mathcal{E}}H^{\alpha}(e),\quad \alpha\geq0.$$ Now, we will show inductively that there exists a constant $C$ that depends only on $\alpha$ and satisfies 
\begin{equation}\label{eq:iden_f_inequality_1}
    \|f\|_{\dot{H}^{\alpha}_{\mathcal{L}}}^2\leq C \sum_{e\in\mathcal{E}}\|f\|_{H^{\alpha}(e)}^2, \quad f\in\dot{H}^{\alpha}_{\mathcal{L}}.
\end{equation} By Theorem~\ref{thm:main_iden_2} Inequality~\eqref{eq:iden_f_inequality_1} holds for $\alpha\in[0,2]$. Assume that the inequality holds for $\alpha\in[2n-2,2n]$ with $n\in\mathbb{N}$, we will show that it holds for $\alpha\in[2n,2n+2]$. Let $\alpha\in[2n,2n+2]$. First, we prove the inequality for $\alpha\in\{2n+1,2n+2\}$. Using the the fact that $\mathcal{L}:\dot{H}^{\alpha}_{\mathcal{L}}\to\dot{H}^{\alpha-2}_{\mathcal{L}}$ is a bijective isometry and the induction hypothesis we have 
\begin{equation}\label{eq:f_H_L_1}
    \begin{split}
        \|f\|_{\dot{H}^{\alpha}_{\mathcal{L}}}^2= \|\mathcal{L}f\|_{\dot{H}^{\alpha-2}_{\mathcal{L}}}^2&\leq C \sum_{e\in\mathcal{E}}\|\mathcal{L}f\|_{H^{\alpha-2}(e)}^2 \\
        &\leq C\sum_{e\in\mathcal{E}} \|\mathcal{L}f\|_{H^{\alpha-3}(e)}^2 +C\sum_{e\in\mathcal{E}}\|\nabla^{\alpha-2}\mathcal{L}f -\check{L}f\|^2_{\mathbb{L}_2(e)} \\&+C \sum_{e\in\mathcal{E}}\|\check{L}f\|^2_{\mathbb{L}_2(e)},
    \end{split}
\end{equation}
where $\check{L}f = \nabla^{\alpha-2}\kappa^2 f - \nabla(\nabla^{\alpha-2}a \nabla f)$.
By Assumption \eqref{assum:regular_coeff} we have
\begin{equation*}
    \begin{split}
        \|\mathcal{L}f\|_{H^{\alpha-3}(e)}^2\leq\;&C\left( \|\kappa^2f\|_{H^{\alpha-3}(e)}^2 + \|\nabla a \nabla f\|_{H^{\alpha-3}(e)}^2 + \|a \nabla^2 f\|_{H^{\alpha-3}(e)}^2\right)\\
        \leq\;& C \|f\|_{H^{\alpha}(e)}^2.
    \end{split}
\end{equation*}
Similarly,
\begin{equation*}
    \begin{split}
        &\mathrel{\hphantom{\leq}}\|\nabla^{\alpha-2}\mathcal{L}f -\check{L}f\|^2_{\mathbb{L}_2(e)}\\&\leq C\sum_{j=1}^{\alpha-2}\|\nabla^{\alpha-2-j}\kappa^2\nabla^jf\|_{\mathbb{L}_2(e)}^2 + C \sum_{j=1}^{\alpha-1}\|\nabla^{\alpha-1-j}a\nabla^{j+1}f\|_{\mathbb{L}_2(e)}^2\\
        & \leq C\sum_{j=1}^{\alpha-2}\|\nabla^jf\|_{\mathbb{L}_2(e)}^2 + C \sum_{j=1}^{\alpha-1}\|\nabla^{j+1}f\|_{\mathbb{L}_2(e)}^2\leq C\|f\|_{H^{\alpha}(e)}^2.
    \end{split}
\end{equation*}
By applying Theorem~\ref{thm:main_iden_2} and Theorem~\ref{thm:main_1} to the operator $\check{L}$ we have 
\begin{equation*}
    \begin{split}
        \|\check{L}f\|_{\mathbb{L}_2(e)}^2 \leq  \|f\|_{\dot{H}_{\check{L}}^{2}}^2
        \leq C\|f\|_{H^{2}(\Gamma)}^2\leq C\|f\|_{H^{\alpha}(\Gamma)}^2
        = C\sum_{e\in\mathcal{E}}\|f\|_{H^{\alpha}(e)}^2.
    \end{split}
\end{equation*}
This proves \eqref{eq:f_H_L_1} for $\alpha\in\{2n+1,2n+2\}$. Assume that $\alpha\in (2n,2n+2)\backslash\{2n+1\}$, and recall that by Lemma~\ref{lem:direct_sum_interpolation_lemma} we have $$\bigoplus_{e\in\mathcal{E}}H^{\alpha}(e)\cong \left[\bigoplus_{e\in\mathcal{E}}H^{\floor{\alpha}}(e), \bigoplus_{e\in\mathcal{E}}H^{\ceil{\alpha}}(e)\right]_{\alpha-\floor{\alpha}},$$
and by \citeA[Theorem~4.36]{lunardi} we have $\dot{H}^{\alpha}\cong\left[\dot{H}^{\floor{\alpha}},\dot{H}^{\ceil{\alpha}}\right]$. By the definition of the interpolation norm (cf. \citeA[Section~1.1]{lunardi}) Inequality~\eqref{eq:f_H_L_1} holds for $\alpha\in (2n,2n+2)\backslash\{2n+1\}$.
Therefore the space $\dot{H}^{\alpha}_{\mathcal{L}}$ is closed in $\bigoplus_{e\in\mathcal{E}}H^{\alpha}(e)$ hence Banach under the norm $\|.\|_{\bigoplus_{e\in\mathcal{E}}H^{\alpha}(e)}$ (by an argument similar to Case~4 in Theorem~\ref{thm:main_1}). The open mapping theorem implies the equivalence of the norms $\|.\|_{\dot{H}^{\alpha}_{\mathcal{L}}}$ and $\|.\|_{\bigoplus_{e\in\mathcal{E}}H^{\alpha}(e)}$ on $\dot{H}^{\alpha}_{\mathcal{L}}$.\\
\medskip

\noindent\textbf{Step 2:} We prove inductively that for every $n\in\{0,\dots,\floor{\frac{\alpha}{2}}\}$ the following identification holds
    \begin{equation}\label{eq:H_calL_1}
        \dot{H}^{\alpha}_{\mathcal{L}}\cong \left\{f\in\bigoplus_{e\in\mathcal{E}}H^{\alpha}(e): \mathcal{L}^{n}f\in \dot{H}^{\alpha-2n}_{\mathcal{L}}, \forall m\in\{0,\dots,n-1\}, \mathcal{L}^mf\in \dot{H}^2_{\mathcal{L}}  \right\}.
    \end{equation} 
    It suffices to show the equality of sets in \eqref{eq:H_calL_1} since the norms are equivalent by Step~1. Let's denote by $\mathcal{A}_{\mathcal{L}}^n$ the left hand side of \eqref{eq:H_calL_1}. For $n=0$, clearly $\dot{H}^{\alpha}=\mathcal{A}^0_{\mathcal{L}}$. Assume that $\alpha\geq3$, and that \eqref{eq:H_calL_1} holds for $n\leq \floor{\frac{\alpha}{2}}-1$, we will prove the equality for $n+1$. We have that $\mathcal{L}^nf\in \dot{H}^{\alpha-2n}_{\mathcal{L}}$ implies $\mathcal{L}^{n+1}f\in \dot{H}^{\alpha-2n-2}_{\mathcal{L}}$ and by the conditions of $n$ we have $\mathcal{L}^nf\in \dot{H}^{\alpha-2n}_{\mathcal{L}}\subset \dot{H}^2_{\mathcal{L}}$. Thus $\mathcal{A}_{\mathcal{L}}^n\subset \mathcal{A}_{\mathcal{L}}^{n+1}$. On the other hand, by the definition of the $\dot{H}_{\mathcal{L}}$ spaces, $\mathcal{L}^{n+1}f\in \dot{H}^{\alpha-2n-2}_{\mathcal{L}}$ implies $\mathcal{L}^nf\in \dot{H}^{\alpha-2n}_{\mathcal{L}}$, hence $\mathcal{A}_{\mathcal{L}}^{n+1}\subset \mathcal{A}_{\mathcal{L}}^{n}$. Therefore Equation~\eqref{eq:H_calL_1} holds.\\
    \medskip
    
    \noindent\textbf{Step 3:} Equation \eqref{eq:H_calL_1} and Theorem~\ref{thm:main_iden_2} applied to $\alpha\in[2,5/2)$ with $n=1$ yields $\dot{H}^{\alpha}_{\mathcal{L}}\cong H^{\alpha}(\Gamma)\cap K(\Gamma)\cap D(\Gamma)$.\\
    \medskip
    
    \noindent\textbf{Step 4:} Finally \eqref{eq:H_calL_main} follows for $\alpha\geq 1$ by setting $n=\floor{\frac{\alpha}{2} - \frac{1}{4}}$ and using Lemma~\ref{lem:H_ladder_1} and Theorems \ref{thm:main_iden_2} and \ref{thm:main_1}, noticing that $\alpha-2n\in(\frac{1}{2},\frac{5}{2})$.
\end{proof}

We have seen in the proof of Theorem~\ref{thm:ident_3} that $$\dot{H}^{\alpha}_{\mathcal{L}}\cong H^{\alpha}(\Gamma)\cap K(\Gamma)\cap D(\Gamma), \quad \forall \alpha\in\left(2,\frac{5}{2}\right).$$
As a result, Corollaries~\ref{corr:well_posedness_calL_2} and~\ref{corr:well_posedness_calL_2_W} can be improved to the following two corollaries.\begin{corr}\label{corr:well_posedness_calL_2_2}
    Assume that $a\in\bigoplus_{e\in\mathcal{E}}C^{1,1}(e)$ and $\kappa^2\in\bigoplus_{e\in\mathcal{E}}C^{0,1}(e)$. Let $(\alpha,\beta)\in(-5,5)\times\left(\left(-\frac{5}{2},\frac{5}{2}\right)\backslash\{\pm\frac{1}{2},\pm\frac{3}{2}\}\right)$ such that $\alpha+\beta\in\left(-\frac{5}{2},\frac{5}{2}\right)\backslash\{\pm\frac{1}{2},\pm\frac{3}{2}\}$, and let 
    $f\in H^{\beta}_{K,D}(\Gamma)$. The fractional PDE
    $$\mathcal{L}^{\alpha/2}u=f$$
    has a unique solution $u\in H^{\beta+\alpha}_{K,D}(\Gamma)$.
    Furthermore, there exists a constant $C$ which depends on $\alpha$ and $\beta$ only that satisfies
    $$\|u\|_{H^{\alpha+\beta}_{K,D}(\Gamma)}\leq C \|f\|_{H^{\beta}_{K,D}(\Gamma)}.$$
\end{corr}
\begin{corr}\label{corr:well_posedness_calL_2_2_W}
    Assume that $a\in\bigoplus_{e\in\mathcal{E}}C^{1,1}(e)$ and $\kappa^2\in\bigoplus_{e\in\mathcal{E}}C^{0,1}(e)$. Let $\alpha\in\left(-2,3\right)$, $\epsilon>0$, and $\mathcal{W}$ be a Gaussian white noise defined on $\Gamma$. The SPDE
    $$\mathcal{L}^{\alpha/2}u=\mathcal{W}$$
    has almost surely a unique solution
    $$u\in H^{\alpha-\frac{1}{2}-\epsilon}_{K,D}(\Gamma).$$
\end{corr}

\bibliographystyle{apacite}
\bibliography{refs}

@article{borovitsky,
  author    = {Borovitskiy, Viacheslav and Azangulov, Iskander and Terenin, Alexander and Mostowsky, Peter and Deisenroth, Marc Peter and Durrande, Nicolas},
  title     = {{M}at\'ern {G}aussian Processes on Graphs. arXiv preprint arXiv:2010.15538},
  year      = {2020},
  eprint    = {stat.ML/2010.15538},
  archivePrefix = {arXiv},
}

@inbook{berk,
author = {Berkolaiko, Gregory},
year = {2017},
pages = {41-72},
booktitle = {Geometric and Computational Spectral Theory},
publisher = {American Mathematical Society},
title = {An elementary introduction to quantum graphs},
fjournal = {Contemporary Mathematics},
journal = {Contemp. Math.},
}

@article{mathcomp_paper,
  author    = {Bolin, David and Kov\'acs, Mih\'aly and Kumar, Vishnu and Simas, Alexandre B.},
  title     = {Regularity and numerical approximation of fractional elliptic differential equations on compact metric graphs},
  journal   = {Mathematics of Computation},
  year      = {2024},
  volume    = {93},
  pages     = {2439-2472}
}

@article{damilya,
      title={Log-{G}aussian {C}ox Processes on General Metric Graphs. arXiv preprint arXiv:2501.18558}, 
      author={David Bolin and Damilya Saduakhas and Alexandre B. Simas},
      year={2025},
      eprint={2501.18558},
      archivePrefix={arXiv},
      primaryClass={stat.ME},
}

@article{bern_gaussian_matern,
  author    = {Bolin, David and Simas, Alexandre B. and Wallin, Jonas},
  title     = {{G}aussian {W}hittle-{M}at\'ern fields on metric graphs},
  journal   = {Bernoulli},
  year      = {2024},
  volume    = {30},
  number    = {2},
  pages     = {1611-1639}
}

@article{markov_paper, 
title={{M}arkov properties of {G}aussian random fields on compact metric graphs}, 
volume={32}, 
number={1},
journal={Bernoulli},
publisher={Bernoulli Society for Mathematical Statistics and Probability},
author={Bolin, David and Simas, Alexandre B. and Wallin, Jonas},
year={2026},
month={Feb},
pages = {153-178}
}

@article{chandler,
  author    = {Chandler-Wilde, Simon N. and Hewett, David P. and Moiola, Andrea},
  title     = {Interpolation of {H}ilbert and {S}obolev Spaces: Quantitative Estimates and Counterexamples},
  journal   = {Mathematika},
  year      = {2015},
  volume    = {61},
  number    = {2},
  pages     = {414-443}
}

@book{demengel,
  author    = {Demengel, Fran{\c{c}}oise and Demengel, Gilbert},
  title     = {Functional Spaces for the Theory of Elliptic Partial Differential Equations},
  publisher = {Springer},
  year      = {2012}
}

@article{di_nezza,
  author    = {Di Nezza, Eleonora and Palatucci, Giampiero and Valdinoci, Enrico},
  title     = {{H}itchhiker's guide to the fractional {S}obolev spaces},
  journal   = {Bulletin des Sciences Math\'ematiques},
  year      = {2012},
  volume    = {136},
  number    = {5},
  pages     = {521-573}
}

@book{evans,
  author    = {Evans, Lawrence C.},
  title     = {Partial Differential Equations},
  edition   = {2},
  publisher = {American Mathematical Society},
  year      = {2010}
}

@book{gilbarg_trudinger,
  author    = {Gilbarg, David and Trudinger, Neil S.},
  title     = {Elliptic Partial Differential Equations of Second Order},
  publisher = {Springer},
  address   = {Berlin, Heidelberg},
  year      = {2001}
}

@article{kuchment,
author = {Peter Kuchment},
title = {Graph models for waves in thin structures},
journal = {Waves in Random Media},
volume = {12},
number = {4},
pages = {R1-R24},
year = {2002},
publisher = {Taylor \& Francis},
}

@book{lunardi,
  author    = {Lunardi, Alessandra},
  title     = {Interpolation Theory},
  series    = {Publications of the Scuola Normale Superiore},
  volume    = {16},
  publisher = {Edizioni della Normale},
  address   = {Pisa},
  year      = {2018}
}

@book{mclean,
  author    = {McLean, William},
  title     = {Strongly Elliptic Systems and Boundary Integral Equations},
  publisher = {Cambridge University Press},
  year      = {2000}
}

@book{miklavcic,
  author    = {Miklav{\v{c}}i{\v{c}}, Milan},
  title     = {Applied Functional Analysis and Partial Differential Equations},
  publisher = {World Scientific Publishing Company},
  year      = {1998}
}

@article{nguyen,
  author    = {Nguyen, Hoai-Minh and Squassina, Marco},
  title     = {Fractional {C}affarelli-Kohn-{N}irenberg inequalities},
  journal   = {Journal of Functional Analysis},
  year      = {2018},
  volume    = {274},
  number    = {9},
  pages     = {2661-2672}
}

@book{zeidler,
  author    = {Zeidler, Eberhard},
  title     = {Applied Functional Analysis: Applications to Mathematical Physics},
  series    = {Applied Mathematical Sciences},
  volume    = {108},
  publisher = {Springer-Verlag},
  address   = {New York},
  year      = {1995}
}

@article{lupu_2016,
title={From loop clusters and random interlacements to the free field},
volume={44},
number={3},
journal={The Annals of Probability},
publisher={Institute of Mathematical Statistics}, author={Lupu, Titus},
year={2016},
month={May}
}

@article{sheffield_2007,
title={{G}aussian free fields for mathematicians},
volume={139},
number={3-4},
journal={Probability Theory and Related Fields},
author={Sheffield, Scott},
year={2007},
month={May},
pages={521–541}
}

@book{berk_2,
  author    = {Berkolaiko, Gregory and Kuchment, Peter},
  title     = {Introduction to Quantum Graphs},
  series    = {Mathematical Surveys and Monographs},
  volume    = {186},
  publisher = {American Mathematical Society},
  address   = {Providence, RI},
  year      = {2013},
  isbn      = {978-0-8218-9211-4},
  mrnumber  = {3013057},
}

@article{rodriguez_2016,
title={A 0-1 law for the massive {G}aussian free field},
volume={169},
number={3-4},
journal={Probability Theory and Related Fields}, 
publisher={Springer Science+Business Media}, 
author={Rodriguez, Pierre-François}, 
year={2016}, 
month={Oct}, 
pages={901–930} 
}

@article{lindgren_bolin_rue_2022, 
title={The {SPDE} approach for {G}aussian and non-{G}aussian fields: 10 years and still running}, 
volume={50}, 
journal={Spatial Statistics}, 
author={Lindgren, Finn and Bolin, David and Rue, H{\aa}vard}, 
year={2022}, 
month={Aug}, 
pages={100599}
}

\end{document}